\documentclass[preprint,12pt]{elsarticle}
\usepackage{amssymb}
\usepackage{amsmath}
\usepackage{amsthm}
\usepackage{appendix,bm,graphicx,mathrsfs,cases,subcaption,hyperref}
\usepackage{color, xcolor}

\journal{B}	

\numberwithin{equation}{section}
\newtheorem{theorem}{Theorem}[section]

\newtheorem{lemma}[theorem]{Lemma}

\newtheorem{remark}{Remark}[section]
\newtheorem{assumption}{Assumption}[section]
\newcommand\diff{\,\mathrm{d}}

\makeatletter
\newcommand{\Rmnum}[1]{\expandafter\@slowromancap\romannumeral #1@}
\makeatother

\begin{document}
	
	\begin{frontmatter}
		
		
		
		\title{High-order mass conserving, positivity plus energy-law preserving schemes and their error estimates for Keller-Segel equations \tnoteref{t1}}
		
		\tnotetext[label1]{This work is partially supported by the Natural Science Foundation of Chongqing
			(No. CSTB2024NSCQ-MSX0221), the National Natural Science Foundation of China
			(No. 12101178), the Natural Science Foundation of Shandong Province (ZR2024QA139).}
		
		\author[label1]{Mingmei Chen}
		\author[label1]{Kun Wang}
		\address[label1]{College of Mathematics and Statistics, Chongqing University, , Chongqing 401331, P.R. China}
		
		\author[label2]{Cong Xie}
		\address[label2]{School of Mathematics and Big Data, Jining University, Jining 273199, China}
	
	\begin{abstract}
		Chemotaxis plays a significant role in numerous physiological processes. The Keller-Segel equation serves as a mathematical model for simulating the phenomenon of cell population aggregation under chemotaxis, possessing physical properties such as mass conservation, positivity of density, and energy dissipation. High-order linear and decoupled schemes for the parabolic-parabolic Keller-Segel chemotaxis model are proposed in this paper,  which satisfy the three physical properties mentioned earlier. Firstly, by applying a logarithmic transformation, the Keller-Segel model is reformulated into its equivalent form that maintains the positivity of cell density regardless of the discrete scheme. Based on this equivalent system, we then propose high-order linear and decoupled numerical schemes using the backward differentiation formula (BDF). Furthermore, through the incorporation of a recovery technique and an energy-law preservation correction (EPC), we ensure that these schemes maintain mass conservation and preserve the original energy-law. Finally, we conduct a rigorous optimal error analysis for the numerical schemes under certain assumptions regarding the regularity of solutions, and some  numerical experiments are also presented to demonstrate their effectiveness.
	\end{abstract}
	
	\journal{XXX}
\end{frontmatter}

	
	
	\section{Introduction}
	\label{sec1}
	Chemotaxis, which plays a vital part in several physiological mechanisms, refers to the phenomenon in which microorganisms sense chemical stimuli signals in the external environment and migrate either towards or away from the source of stimulation. The chemotaxis model was first introduced by Patlak~\cite{bib01} in the 1950s. Subsequently, in the 1970s, Keller and Segel~\cite{bib02,bib03,bib04} established the Keller-Segel chemotaxis model, which reveals the changing process between cell or microbial population density and chemoattractant concentration. Our primary focus in this work is on the parabolic-parabolic Keller-Segel chemotaxis model:
	\begin{subequations}\label{eq.1.1}
		\begin{align}
			&\varepsilon \partial_t c =\varDelta c-\alpha c+\beta\rho, &\mathrm{in}\ \Omega\times (0,T],\label{eq.1.1a}\\
			&\partial_t \rho =\varDelta \rho - \nabla \cdot(\gamma\rho\nabla c), &\mathrm{in}\ \Omega\times (0,T],\label{eq.1.1b}\\
			&\frac{\partial c}{\partial \nu}=0, \quad\frac{\partial \rho}{\partial \nu}-\gamma\rho\frac{\partial c}{\partial \nu}=0, &\mathrm{on}\ \partial\Omega\times (0,T], \label{eq.1.1c}\\
			&c(\mathbf{x},0)=c_0(\mathbf{x}),\quad \rho(\mathbf{x},0)=\rho_0(\mathbf{x}), &\mathrm{in}\ \Omega,\label{eq.1.1d}
		\end{align}
	\end{subequations}
	where $\Omega\subset\mathbb{R}^2$ is a bounded domain with a sufficiently smooth boundary $\partial\Omega$, and $T>0$ is a finite time. The unknown functions $\rho=\rho(\mathbf{x}, t)=\rho(x, y, t)$ and $c=c(\mathbf{x}, t)=c(x, y, t)$ represent the cell density function and the chemoattractant concentration function, respectively. $\rho_0(\mathbf{x})>0$ and $c_0(\mathbf{x})>0$ are given initial data. The parameters $\varepsilon$, $\alpha$, $\beta$ and $\gamma$ are all positive. Specifically, $\varepsilon$ describes the response rate of the chemoattractant concentration to the cell density, $\alpha$ is the reaction coefficient of the chemoattractant, $\beta$ denotes the growth rate of the cells, and $\gamma$ is the chemotactic sensitivity coefficient related to the chemoattractant concentration. Additionally, $\nu$ is the unit outward normal vector to the boundary $\partial\Omega$.
	
	Significantly, the system \eqref{eq.1.1} possess several essential properties:
	
	1. Mass conservation:
	\begin{equation}
		\int_\Omega\rho(\mathbf{x},t)\diff\mathbf{x}=\int_\Omega\rho(\mathbf{x},0)\diff\mathbf{x}.\label{eq.1.2}
	\end{equation}
	
	2. Positivity:
	\begin{equation}
		\rho(\mathbf{x},t)>0.\notag
	\end{equation}
	
	3. Energy dissipation law:
	\begin{equation}
		\frac{\diff E(\rho,c)}{\diff t}=-\int_{\Omega}\rho|\nabla(\log\rho-c)|^2+\varepsilon|\partial_t c|^2\diff \mathbf{x},\label{eq.1.3}
	\end{equation}
	where the free energy $E$ is defined as
	\begin{equation}
		E(\rho,c)=\int_{\Omega}\rho\log\rho-\rho-\rho c+\frac12\alpha c^2+\frac12|\nabla c|^2\diff\mathbf{x}.\label{eq.1.4}
	\end{equation}
	If the chemoattractant diffuses much faster than organisms, the appropriate mathematical model to represent this scenario is to set $\varepsilon=0$. In such case, we refer to the system of equations described by \eqref{eq.1.1} as the parabolic-elliptic system. This specific system has also attracted considerable attention and has been widely investigated in the existing literature. For further reading and detailed analysis, one may refer to \cite{bib05,bib06,bib07,bib08}. We consider the case with $\varepsilon=\alpha=\beta=\gamma=1$ in this research.
	
	Researchers have extensively investigated numerous analytical results concerning the solutions of the classical Keller-Segel model, including those related to existence, uniqueness, boundedness, and blow-up phenomena~\cite{bib09,bib10}. For example, Calvez and Corrias in \cite{bib11} derive a critical mass threshold, beneath which solutions are guaranteed to exist globally. It is shown in \cite{bib12} that when the initial mass of cells exceeds a critical value, blow-up occurs within finite time. Certainly, it does not occur in real scenarios, which just indicates that the organisms will gather in a very small area. To prevent the occurrence of non-physical blow-up, a range of modifications to the classical model have been introduced. One can refer to \cite{bib13,bib14,bib15,bib16} and the references therein for more theoretical results.
	
	At the same time, many numerical methods for the Keller-Segel model have attracted extensive attention and research in recent years. Among them are the finite difference method (FDM), referenced in \cite{bib05}, the finite volume method (FVM) as studied in \cite{bib17,bib18,bib19,bib20}, the finite element method (FEM) explored in \cite{bib21,bib22,bib23}, the discontinuous Galerkin method (DGM) discussed in \cite{bib24,bib25,bib26,bib33}, as well as other approaches detailed in \cite{bib27,bib28,bib29}.
	
	When designing numerical schemes for the Keller-Segel model, one of the challenges is to maintain the physical properties of the continuous model at the discrete level. Numerous studies have been dedicated to addressing this issue. For example, A moving mesh finite element method that preserves positivity was presented in \cite{bib22}. The work in \cite{bib30} proposed a new upwind DGM designed to maintain mass conservation, positivity, and energy stability properties. Shen and Xu~\cite{bib31} presented a fully discrete scheme by reformulating the term $\varDelta\rho$ as $\nabla\cdot(\frac1{f''(\rho)}\nabla f'(\rho))$ with $f''(\rho)=\frac1{\rho}$. They proved that the first-order scheme possesses properties including mass conservation, bound preservation, uniquely solvability, and energy dissipation, while the second-order scheme retains the first three of these properties. In \cite{bib32}, a new method for constructing high-order schemes that preserve positivity or bound and ensure unconditionally energy stable was proposed. This method combines the scalar auxiliary variable (SAV) approach with function transformation. Ding, Wang and Zhou~\cite{bib38} proposed a second-order scheme employing a modified Crank-Nicolson approach. They proved that this scheme preserves essential structural properties, including unique solvability, positivity preservation, mass conservation, and the original energy dissipation. Additionally, an estimate of the optimal convergence rate is given for the proposed method. Xu an Fu~\cite{bib16} construct a decoupled linear, block-centered finite difference method, demonstrating its mass conservation and second-order temporal and spatial convergence. In a separate study, Wang, Liu and Feng~\cite{bib34} applied a log-transformation to ensure positivity preservation, and then utilized a recovery technique to maintain mass conservation. They proved the optimal convergence order in the $L^2$-norm for a first-order scheme. But, how to design higher-order schemes preserving the original energy-law  and prove their optimal convergence order is still a challenging problem.
	
	In this paper, we propose high-order linear and decoupled numerical schemes that not only conserve mass and maintain positivity, but also preserve the original energy-law for the model~\eqref{eq.1.1}. Following the approach in \cite{bib34}, this work initially uses a log-transformation to the cell density to ensure positivity and rewrites the system~\eqref{eq.1.1} in an equivalent form. Then, high-order temporal discrete schemes for the equivalent model \eqref{eq.2.3} are constructed using the backward differentiation formula (BDF) method, combined with a recovery technique to ensure mass conservation. Additionally, we construct an energy-law preservation correction (EPC) technique, which enables us to ensure that the original energy-law of system \eqref{eq.1.1} is preserved. Finally, under the Assumptions~\ref{as.2.1} and \ref{as.3.1}, we establish the following optimal convergence order for the $k$th-order  scheme $(k=1, \cdots, 5)$:  for any $n\ge k$, there holds
	\begin{align*}
		&\|c(t_n)-c^n\|_{L^2(\Omega)}+\|\rho(t_n)-\rho^n\|_{L^2(\Omega)}\\
		&+\|\nabla (c(t_n)-c^n)\|_{L^2(\Omega)}+\|\nabla(\rho(t_n)-\rho^n)\|_{L^2(\Omega)}\le C\tau^k,
	\end{align*}
	where $\tau$ denotes the time step, $C$ is general positive constant, and $c^n$ and $\rho^n$ represent the approximation solutions of $c(t_n)$ and $\rho(t_n)$, respectively.
	
	The paper is organized as follows. In Section \ref{sec.2}, we introduce some essential notations and key inequalities, present the equivalent equation corresponding to equation \eqref{eq.1.1}, and establish some necessary assumptions for the following analysis. In Section \ref{sec.3}, we propose $k$th-order numerical schemes for the equivalent system. These schemes are mass conservative, positivity plus original energy-law preserving. Additionally, this section presents the main results concerning the $k$th-order convergence of these schemes. In Section \ref{sec.4}, we rigorously prove the theorem on the convergence order through the application of mathematical induction. Section \ref{sec.5} offers numerical examples to validate the derived theoretical results. Finally, Section \ref{sec.6} is devoted to some final remarks.
	
	\section{Preliminaries}\label{sec.2}
	Below, we first introduce a number of notations and inequalities that will be frequently used in the subsequent sections. Following this presentation, we reformulate the Keller-Segel model into an equivalent system by employing a logarithmic transformation.
	
	In the sequel, $C$ (whether subscripted or not) denotes a general positive constant. Although this constant may differ in various contexts, it consistently remains independent of the time step size $\tau$. We omit the dependence on $\mathbf{x}$ for all functions when there is no confusion.
	
	For $1\le p\le+\infty$ and $m\in\mathbb{N}^+$, we represent the Lebesgue space with $L^p(\Omega)$ and the Sobolev space with $W^{m,p}(\Omega)$. $\|\cdot\|_{L^p}$ and $\|\cdot\|_{W^{m,p}}$ denote the norms in these spaces, which are defined in the classical manner. In the special case where $p=2$, the Sobolev space $W^{m,2}(\Omega)$ is written as $H^m(\Omega)$, with its norm denoted by $\|\cdot\|_m$. In particular, $\|\cdot\|_0$ signifies the norm in $L^2(\Omega)$. For simplicity, $\|\cdot\|_\infty$ represents $\|\cdot\|_{L^\infty}$, and $(\cdot,\cdot)$ denotes the inner product in $L^2(\Omega)$.
	
	For simplicity, we define
	\begin{align*}
		a(u,v)&=(\nabla u,\nabla v),\quad\forall u,v\in H^1(\Omega),\\
		b(u,v,w)&=(\nabla u\cdot\nabla v,w),\quad\forall u,v\in H^1(\Omega), w\in L^2(\Omega).
	\end{align*}
	By applying the Hölder's inequality, one can derive the following inequalities for the previously defined trilinear term $b(\cdot,\cdot,\cdot)$ (see \cite{bib34}):
	\begin{subequations}\label{eq.2.1}
		\begin{align}
			|b(u,v,w)| \le \|\nabla u\|_0 \|\nabla v\|_{L^4} \|w\|_{L^4},\label{eq.2.1a}\\
			|b(u,v,w)| \le \|\nabla u\|_{L^4} \|\nabla v\|_0 \|w\|_{L^4},\label{eq.2.1b}\\
			|b(u,v,w)| \le \|\nabla u\|_{L^4}\|\nabla v\|_{L^4}\|w\|_0\label{eq.2.1c}.
		\end{align}
	\end{subequations}
	Moreover, there hold the Gagliardo-Nirenberg inequality and Agmon's inequality \cite{bib35}:
	\begin{subequations}\label{eq.2.2}
		\begin{align}
			\|w\|_{L^4}&\le C \|w\|^{\frac12}_0 \|\nabla w\|^{\frac12}_0, \quad\forall w \in H^1(\Omega),\label{eq.2.2a}\\
			\|w\|_\infty&\le C \|w\|^{\frac12}_0 \|\varDelta w\|^{\frac12}_0,\quad\forall w\in L^{\infty}(\Omega) \cap H^2(\Omega).\label{eq.2.2b}
		\end{align}
	\end{subequations}
	To keep the positivity of the cell density $\rho$, we introduce a new variable $u=\log(\rho)$ and, by substituting it into the original equation \eqref{eq.1.1}, obtain its equivalent form as follows:
	\begin{subequations}\label{eq.2.3}
		\begin{align}
			&\partial_t c =\varDelta c-c+\rho, &\mathrm{in}\ \Omega\times(0,T],\label{eq.2.3a}\\
			&\partial_t u=\varDelta u+\nabla u\cdot\nabla u-\nabla u\cdot\nabla c-\varDelta c,&\mathrm{in}\ \Omega\times(0,T],\label{eq.2.3b}\\
			&\rho=\exp(u),&\mathrm{in}\ \Omega\times(0,T],\label{eq.2.3c}\\
			&\frac{\partial c}{\partial\nu}=\frac{\partial u}{\partial\nu}=0,&\mathrm{on}\ \partial\Omega\times(0,T],\label{eq.2.3d}\\
			&c(\mathbf{x},0)=c_0(\mathbf{x}),u(\mathbf{x},0)=u_0(\mathbf{x})=\log(\rho_0(\mathbf{x})),\quad &\mathrm{in}\ \Omega.\label{eq.2.3e}
		\end{align}
	\end{subequations}
	From the equivalent form provided above, it becomes evident that the cell density $\rho$ has been transformed into an exponential function, ensuring that $\rho$ remains positive irrespective of the numerical method employed for its computation. Furthermore, in contrast to commonly used positivity preserving numerical methods, this method is independent of the discrete scheme and avoids the introduction of any nonlinear stabilizing term. Consequently, it simplifies the schemes and facilitates the derivation of optimal convergence orders in the subsequent section.
	
	To deduce error estimates, we make the following assumptions concerning $(c_0,\rho_0)$ and the exact solution $(c,u)$ of equation \eqref{eq.2.3}.
	\begin{assumption}\label{as.2.1}
		Assume that $\rho_0>0$ and the initial mass, given by  $\int_{\Omega}\rho_0\diff\mathbf{x}$, is below a specific critical threshold, ensuring that the system will not experience finite-time blow-up. Additionally, the initial value $c_0$ satisfies $\int_{\Omega}c_0\diff\mathbf{x}>0$. And for any $t\in (0,T]$, the exact solution $(c,u)$ of equation \eqref{eq.2.3} satisfies the following regularity condition:
		\begin{align*}
			&\|c_0\|_2^2+\|c(t)\|_2^2+\|c_t(t)\|_0^2+\sum_{i=1}^{k+1}\bigg(\int_0^t\|\partial_{\tilde{t}}^ic(\tilde{t})\|_0^4+\|\nabla\partial_{\tilde{t}}^ic(\tilde{t})\|_0^4+\|\varDelta\partial_{\tilde{t}}^ic(\tilde{t})\|_0^4\diff \tilde{t}\bigg)\\
			&+\|u_0\|_2^2+\|u(t)\|_2^2+\sum_{i=1}^{k+1}\bigg(\int_0^t\|\partial_{\tilde{t}}^iu(\tilde{t})\|_0^4+\|\nabla\partial_{\tilde{t}}^iu(\tilde{t})\|_0^4\diff\tilde{t}\bigg)+\int_0^t\|\varDelta\partial_{\tilde{t}}^ku(\tilde{t})\|_0^2\diff\tilde{t}\le C.
		\end{align*}
	\end{assumption}
	Moreover, we can derive from \eqref{eq.2.2}, \eqref{eq.2.3c} and Assumption ~\ref{as.2.1} that
	\begin{align}
		\|\rho_0\|_2^2+\|\rho(t)\|_2^2+\sum_{i=1}^{k+1}\bigg(\int_0^t\|\partial_{\tilde{t}}^i\rho(\tilde{t})\|_0^4+\|\nabla\partial_{\tilde{t}}^i\rho(\tilde{t})\|_0^4\diff\tilde{t}\bigg)\le C.\label{eq.2.4}
	\end{align}
	
	\section{Schemes and main results}\label{sec.3}
	In this section, we develope high-order schemes for the equivalent model \eqref{eq.2.3} using BDF, incorporating a recovery technique and an EPC, thereby ensuring that these schemes are mass conservative and original energy-law preserving. The main results concerning the $k$th-order convergence accuracy are presented in Theorem \ref{th.3.1}.
	
	Let $N\in\mathbb{N}^+$ and consider a partition of $[0,T]$ such that $0=t_0\le t_1\le\cdots\le t_n\le t_{n+1}\cdots\le t_N=T$, where $\tau=T/N$ is a fixed time step size, and $t_n=n\tau$ for $n=0,1,\cdots,N-1$. We define the $k$-step backward differentiation operator as $D_{k\tau}v^{n+1}:=\frac{\alpha_kv^{n+1}-A_k(v^n)}{\tau}$ for any sequence $\{v^n\}_{n=k-1}^{N-1}$, where $\alpha_k$, $A_k(v^n)$ and the extrapolation operator $B_k(v^n)$ are defined as follows:
	
	First-order $(k=1)$:
	\begin{align}\label{k=1}
		\alpha_1=1,{}A_1(v^n)=v^n,{}B_1(v^n)=v^n;
	\end{align}
	
	Second-order $(k=2)$:
	\begin{align}\label{k=2}
		\alpha_2=\frac32,{}A_2(v^n)=2v^n-\frac12v^{n-1},{}B_2(v^n)=2v^n-v^{n-1};
	\end{align}
	
	Third-order $(k=3)$:
	\begin{align}\label{k=3}
		\begin{aligned}
			\alpha_3=\frac{11}{6},{}A_3(v^n)={}&3v^n-\frac32v^{n-1}+\frac13v^{n-2},\\ B_3(v^n)={}&3v^n-3v^{n-1}+v^{n-2};
		\end{aligned}
	\end{align}
	
	Fourth-order $(k=4)$:
	\begin{align}\label{k=4}
		\begin{aligned}
			\alpha_4=\frac{25}{12},A_4(v^n)={}&4v^n-3v^{n-1}+\frac43v^{n-2}-\frac14v^{n-3},\\ B_4(v^n)={}&4v^n-6v^{n-1}+4v^{n-2}-v^{n-3};
		\end{aligned}
	\end{align}
	
	Fifth-order $(k=5)$:
	\begin{align}\label{k=5}
		\begin{aligned}
			\alpha_5=\frac{137}{60},A_5(v^n)={}&5v^n-5v^{n-1}+\frac{10}3v^{n-2}-\frac54v^{n-3}+\frac15v^{n-4},\\
			B_5(v^n)={}&5v^n-10v^{n-1}+10v^{n-2}-5v^{n-3}+v^{n-4}.
		\end{aligned}
	\end{align}
	
	Next, we construct $k$th-order schemes to approximate the equivalent system \eqref{eq.2.3} in a unified form. Given the initial values $(\bar{c}^i, u^i, \rho^i, c^i)$ that satisfy  $\int_{\Omega}\rho^i\diff \mathbf{x}=\int_{\Omega}\rho_0\diff \mathbf{x}$ for $i=0,\cdots,k-1$, we determine ($\bar{c}^{n+1}$, $u^{n+1}$, $\bar{\rho}^{n+1}$, $\lambda^{n+1}$, $\rho^{n+1}$, $\mu^{n+1}$, $c^{n+1}$) for $n=k-1,\cdots,N-1$, by following these steps.
	
	$\textbf{Step 1}$. Find $\bar{c}^{n+1}$ by solving
	\begin{equation}
		D_{k\tau}\bar{c}^{n+1}=\varDelta\bar{c}^{n+1}- \bar{c}^{n+1}+B_k(\rho^n)\label{eq.3.1},
	\end{equation}
	with the boundary condition being $\partial\bar{c}^{n+1}/\partial\nu=0$;
	
	$\textbf{Step 2}$. Find $u^{n+1}$ by solving
	\begin{equation}
		D_{k\tau}u^{n+1}=\varDelta u^{n+1}+\nabla u^{n+1} \cdot \nabla B_k(u^n)-\nabla u^{n+1} \cdot \nabla B_k(c^n)-\varDelta \bar{c}^{n+1},\label{eq.3.2}
	\end{equation}
	with the boundary condition being $\partial u^{n+1}/\partial \nu=0$;
	
	$\textbf{Step 3}$. Find $\rho^{n+1}$ by
	\begin{equation}
		\rho^{n+1}=\lambda^{n+1} \bar{\rho}^{n+1};\label{eq.3.3}
	\end{equation}
	with
	\begin{equation}\label{eq.3.4}
		\bar{\rho}^{n+1}=\exp(u^{n+1}),\quad\lambda^{n+1}=\frac{\int_\varOmega \rho^n \diff\mathbf{x}}{\int_\varOmega \bar{\rho}^{n+1} \diff\mathbf{x}};
	\end{equation}
	
	$\textbf{Step 4}$. Find $\mu^{n+1}$ by solving
	\begin{equation}\label{eq.3.5}
		D_{k\tau}E^{n+1}=-\int_{\Omega}\Big[\rho^{n+1}|\nabla(\log \rho^{n+1} - \bar{c}^{n+1})|^2+|D_{k\tau}\bar{c}^{n+1}|^2\Big] \diff\mathbf{x},
	\end{equation}
	where	
	\begin{equation*}
		E^{n+1}=\int_{\Omega}\Big[\rho^{n+1}\log(\rho^{n+1})-\rho^{n+1} -\rho^{n+1}\bar{c}^{n+1}+\frac12\mu^{n+1} (\bar{c}^{n+1})^2+\frac12\mu^{n+1}|\nabla \bar{c}^{n+1}|^2\Big]\diff\mathbf{x}.
	\end{equation*}
	
	$\textbf{Step 5}$. Find $c^{n+1}$ by
	\begin{equation}
		c^{n+1}=\sqrt{\mu^{n+1}} \bar{c}^{n+1}.\label{eq.3.6}
	\end{equation}
	In these numerical schemes, $\bar{c}^0=c^0=c_0,~u^0=u_0,~\rho^0=\rho_0$, $(\bar{c}^{n+1}, c^{n+1})$, $u^{n+1}$ and $(\bar{\rho}^{n+1}, \rho^{n+1})$ are the temporal semi-discrete $k$th-order approximations of $c(t_{n+1})$, $u(t_{n+1})$ and $\rho(t_{n+1})$ respectively, and ($\lambda^{n+1}, \mu^{n+1}$) is the $k$th-order approximation of the constant $1$. These approximations will be proved in the next section and it will be shown that $\mu^{n+1}\in\mathbb{R}^+$ is solvable if $\tau$ is sufficiently small. In Step 3, we incorporate a recovery technique to ensure mass conservation within the system. In Steps 4-5, we introduce an EPC to guarantee that these schemes preserve the energy-law \eqref{eq.1.3}.
	\begin{remark}
		The numerical solutions for the first step of the second-order scheme can be obtained by using the first-order scheme. Similarly, when dealing with the $k$th-order scheme (where $k=3, 4, 5$), the initial $k-1$ steps can be initialized using a $k-1$th-order Runge-Kutta method.
	\end{remark}
	\begin{remark}
		From the above schemes, it is evident that a decoupled and linear system is solved exclusively at Steps 1 and 2. Steps 3-5, which consist solely of assignment operations, can be implemented efficiently.
	\end{remark}
	
	To facilitate subsequent error analysis, we present the following assumptions about the initial values for the first $k-1$ steps.
	\begin{assumption}\label{as.3.1}
		Assume that for $i=1,\cdots,k-1$, the values $(\bar{c}^i,u^i, \bar{\rho}^i, \lambda^i, \rho^i, \mu^i, c^i)$ have been computed using an appropriate initialization procedure, ensuring that
		\begin{align*}
			&\|c(t_i)-c^i\|_0^2+\|c(t_i)-\bar{c}^i\|_0^2+\|\rho(t_i)-\rho^i\|_0^2+\|\rho(t_i)-\bar{\rho}^i\|_0^2\\
			&+\|u(t_i)-u^i\|_0^2+|1-\lambda^i|^2+|1-\mu^i|^2\le C\tau^{2k},\\
			&\|\nabla \big(c(t_i)-c^i\big)\|_0^2+\|\nabla\big(c(t_i)-\bar{c}^i\big)\|_0^2+\|\nabla\big(\rho(t_i)-\rho^i\big)\|_0^2\\
			&+\|\nabla\big(\rho(t_i)-\bar{\rho}^i\big)\|_0^2+\|\nabla\big(u(t_i)-u^i\big)\|_0^2\le C\tau^{2k}.
		\end{align*}
	\end{assumption}
	
	The theorem below presents the main results related to the convergence order, with its proof detailed in the subsequent section.
	\begin{theorem}\label{th.3.1}
		Under the Assumptions \ref{as.2.1} and \ref{as.3.1}, there exists a sufficiently small positive constant $\tau^*$ such that, for any $\tau$ satisfying $\tau\le\tau^*$ and for $n=k-1,\cdots,N-1$, the solutions obtained by the schemes \eqref{eq.3.1}-\eqref{eq.3.6} fulfill
		\begin{align}
			\|\bar{e}_c^{n+1}\|_0^2+\tau\sum_{q=k-1}^n\|\nabla\bar{e}_c^{k+1}\|_0^2\le{}& C\tau^{2k},\label{eq.3.8}\\
			\|\nabla\bar{e}_c^{n+1}\|_0^2+ \tau\sum_{q=k-1}^n\|\varDelta\bar{e}_c^{k+1}\|_0^2\le{}& C\tau^{2k},\label{eq.3.9}\\
			\|e_u^{n+1}\|_0^2+\tau\sum_{q=k-1}^n\|\nabla e_u^{n+1}\|_0^2\le{}& C\tau^{2k},\label{eq.3.10}\\
			\|\nabla e_u^{n+1}\|_0^2+\tau\sum_{q=k-1}^n\|\varDelta e_u^{n+1}\|_0^2\le{}& C\tau^{2k},\label{eq.3.11}\\
			\|\bar{e}_\rho^{n+1}\|_0^2+\|\nabla\bar{e}_\rho^{n+1}\|_0^2+\|e_\rho^{n+1}\|_0^2+\|\nabla e_\rho^{n+1}\|_0^2+|1-\lambda^{n+1}|^2\le{}& C\tau^{2k},\label{eq.3.12}\\			
			\|e_c^{n+1}\|_0^2+\|\nabla e_c^{n+1}\|_0^2+|1-\mu^{n+1}|^2\le{}& C\tau^{2k},\label{eq.3.13}\\
			\|\bar{c}^{n+1}\|_2^2+\|c^{n+1}\|_2^2+\|u^{n+1}\|_2^2+\|\rho^{n+1}\|_2^2\le{}& C.\label{eq.3.14}
		\end{align}
	\end{theorem}
	With the above theorem, we can confirm that these schemes satisfy the following three properties.
	\begin{theorem}\label{th.3.2}
		Assuming Theorem \ref{th.3.1} holds, given $\rho^i>0$ such that $\int_{\Omega}\rho^i\diff \mathbf{x}=\int_{\Omega}\rho_0\diff \mathbf{x}$ for $i=0,1,\cdots,k-1$, then $\mu^{n+1}$ is solvable in $\mathbb{R}^+$ and the numerical schemes \eqref{eq.3.1}-\eqref{eq.3.6} possess the following properties for $n=k-1,\cdots,N-1$:\\
		1. Mass conservation: $\int_{\Omega}\rho^{n+1}\diff \mathbf{x}=\int_{\Omega}\rho_0\diff \mathbf{x}$;\\
		2. Positivity preservation: $\rho^{n+1}>0$;\\
		3. Energy-law preservation:
		\begin{equation*}
			D_{k\tau}E^{n+1}=-\int_{\Omega}\rho^{n+1}|\nabla(\log \rho^{n+1} - \bar{c}^{n+1})|^2+|D_{k\tau}\bar{c}^{n+1}|^2\diff \mathbf{x}\le0,
		\end{equation*}
		where $E^{n+1}$ is defined in \eqref{eq.3.5}.
	\end{theorem}
	\begin{proof}
		For the proof of mass conservation, please refer to \eqref{mass conservation} as presented in the proof process of Theorem \ref{th.3.1}. From \eqref{eq.3.4}, we can infer $\bar{\rho}^{n+1}>0$ and $\lambda^{n+1}>0$. Therefore, using \eqref{eq.3.3}, we can deduce $\rho^{n+1}>0$. Finally, the energy-law preservation can be easily derived from \eqref{eq.3.5}.
	\end{proof}
	
	\section{Error Analysis}\label{sec.4}
	This section is dedicated to proving Theorem \ref{th.3.1}, based on Assumptions~\ref{as.2.1} and \ref{as.3.1}.
	
	For notational simplicity, we denote
	\begin{align}
		\bar{e}_c^n:={}&c(t_n)-\bar{c}^n,\quad e_c^n:=c(t_n)-c^n,\quad e_u^n:=u(t_n)-u^n,\notag\\
		\bar{e}_\rho^n:={}&\rho(t_n)-\bar{\rho}^n,\quad e_\rho^n:=\rho(t_n)-\rho^n.\notag
	\end{align}
	
	Next, we derive the error equations, which are essential for analyzing the error between the exact and numerical solutions. For $n=k-1,\cdots,N-1$, based on \eqref{eq.2.3a} and \eqref{eq.2.3b}, the exact solutions $c(t_{n+1})$, $\rho(t_{n+1})$ and $u(t_{n+1})$ satisfy the following variational formulations:
	\begin{align}
		D_{k\tau}c(t_{n+1})={}&\varDelta c(t_{n+1})-c(t_{n+1})+\rho(t_{n+1})-\big(c_t(t_{n+1})-D_{k\tau}c(t_{n+1})\big),\label{eq.4.1}\\
		D_{k\tau}u(t_{n+1})={}&\varDelta u(t_{n+1})+\nabla u(t_{n+1})\cdot\nabla u(t_{n+1})-\nabla u(t_{n+1})\cdot\nabla c(t_{n+1})\notag\\
		&-\varDelta c(t_{n+1})-\big(u_t(t_{n+1})-D_{k\tau}u(t_{n+1})\big).\label{eq.4.2}
	\end{align}
	By subtracting \eqref{eq.3.1} from \eqref{eq.4.1} and \eqref{eq.3.2} from \eqref{eq.4.2} yields the subsequent error equations:
	\begin{align}
		D_{k\tau}\bar{e}_c^{n+1}-\varDelta\bar{e}_c^{n+1}+\bar{e}_c^{n+1} =B_k(e_\rho^n)+Q_k^{n+1}(\rho)-R_k^{n+1}(c),\label{eq.4.3}
	\end{align}
	and
	\begin{align}\label{eq.4.4}
		D_{k\tau}e_u^{n+1}-\varDelta e_u^{n+1}={}&\nabla u(t^{n+1})\cdot\nabla\big(Q_k^{n+1}(u)+B_k(e_u^n)\big)+\nabla e_u^{n+1}\cdot\nabla B_k(u^n)\notag\\
		&-\nabla u(t^{n+1})\cdot\nabla\big(Q_k^{n+1}(c)+B_k(e_c^n)\big)
		-\nabla e_u^{n+1}\cdot\nabla B_k(c^n)\notag\\
		&-\varDelta\bar{e}_c^{n+1}-R_k^{n+1}(u),
	\end{align}
	where $R_k^{n+1}(v)$ and $Q_k^{n+1}(v)$ represent the truncation error functions defined as
	\begin{align}
		R_k^{n+1}(v):={}&v_t(t_{n+1})-D_{k\tau}v(t_{n+1})\notag\\
		={}&\frac1\tau\sum_{i=1}^k\bigg(r_i\int_{t^{n+1-i}}^{t^{n+1}}(t-t^{n+1-i})^k\partial_t^{k+1}v(t)\diff t\bigg),\label{eq.4.5}\\
		Q_k^{n+1}(v):={}&v(t_{n+1})-B_k(v(t_n))\notag\\
		={}&\sum_{i=1}^k\bigg(s_i\int_{t^{n+1-i}}^{t^{n+1}}(t-t^{n+1-i})^{k-1}\partial_t^kv(t)\diff t\bigg),\notag
	\end{align}
	with $r_i$ and $s_i$ being fixed constants that are determined by the backward differentiation operator $D_{k\tau}$ and the extrapolation operator $B_k$. Using Hölder's inequality, we can easily obtain
	\begin{subequations}\label{eq.4.6}
		\begin{align}
			\|R_k^{n+1}(v)\|_0^2\le& C\tau^{2k-1}\int_{t_{n+1-k}}^{t_{n+1}} \|\partial_t^{k+1}v\|_0^2 \diff t,\label{eq.4.6a}\\
			\|Q_k^{n+1}(v)\|_0^2\le& C\tau^{2k-1}\int_{t_{n+1-k}}^{t_{n+1}} \|\partial_t^kv\|_0^2\diff t.\label{eq.4.6b}
		\end{align}
	\end{subequations}
	Concerning the backward difference operator $D_{k\tau}$, there exists the following lemma, which plays a significant role in error analysis.
	\begin{lemma}[{\cite{bib36}, Lemma 1}]\label{le.4.1}
		For $k=1, \cdots, 5$, there exist constants $\sigma_k$ satisfying $0\le \sigma_k<1$, a symmetric and positive definite matrix $G=(g_{ij})\in \mathbb{R}^{k\times k}$, and real numbers $\delta_0,\dots,\delta_k$ such that the following equality holds:
		\begin{align*}
			\big(D_{k\tau}v^{n+1},\tau(v^{n+1}-\sigma_kv^n)\big)={}&\sum_{i,j=1}^kg_{ij}(v^{n+1+i-k},v^{n+1+j-k})\\\notag
			-{}&\sum_{i,j=1}^kg_{ij}(v^{n+i-k},v^{n+j-k})+\|\sum_{i=0}^k\delta_iv^{n+1+i-k}\|_0^2,
		\end{align*}
		where the smallest feasible values of $\sigma_k$ are given by
		\begin{align*}
			\sigma_1=\sigma_2=0,\quad\sigma_3=0.0836,\quad\sigma_4=0.2878,\quad\sigma_5=0.8160,
		\end{align*}
		and $\alpha_k$, $A_k$ are defined the same as in \eqref{k=1}-\eqref{k=5}.
	\end{lemma}
	For the matrix $G=(g_{ij})$ mentioned above, since it is symmetric and positive definite, the following conclusions can be deduced:
	\begin{align*}
		\lambda_{\min}\sum_{i=1}^k\|v^{n+1+i-k}\|_0^2\le\sum_{i,j=1}^kg_{ij}(v^{n+1+i-k},v^{n+1+j-k})\le\lambda_{\max}\sum_{i=1}^k\|v^{n+1+i-k}\|_0^2,
	\end{align*}
	where $\lambda_{\min},\lambda_{\max}>0$ denote the smallest and largest eigenvalues of $G=(g_{ij})$, respectively. This can be further inferred
	\begin{align}
		\lambda_{\min}\|v^{n+1}\|_0^2\le\sum_{i,j=1}^kg_{ij}(v^{n+1+i-k},v^{n+1+j-k}).\label{eq.4.7}
	\end{align}
	
	By analyzing Steps 1-5 within the discrete schemes \eqref{eq.3.1}-\eqref{eq.3.6}, we can sequentially obtain error estimates for $\bar{c}^{n+1}$, $u^{n+1}$, $\bar{\rho}^{n+1}$, $\lambda^{n+1}$, $\rho^{n+1}$, $\mu^{n+1}$, and $c^{n+1}$. These error estimates are summarized in Lemmas \ref{le.4.2}-\ref{le.4.6}. Subsequently, by employing these lemmas in mathematical induction, we deduce Theorem \ref{th.3.1}.
	\begin{lemma}\label{le.4.2}
		Under the Assumption \ref{as.2.1}, for all $n=k-1,\cdots,N-1$, we have
		\begin{align}
			\|\bar{e}_c^{n+1}\|_0^2+\tau\sum_{q=k-1}^n\|\nabla\bar{e}_c^{q+1}\|_0^2\le {}&C\tau^{2k}+C\tau\sum_{q=k-1}^n\|\bar{e}_c^{q}\|_1^2+C\tau\sum_{q=0}^n\|e_\rho^q\|_0^2,\label{eq.4.8}\\
			\|\nabla \bar{e}_c^{n+1}\|_0^2+ \tau\sum_{q=k-1}^n\|\varDelta\bar{e}_c^{q+1}\|_0^2 \le{}&C\tau^{2k}+C\tau\sum_{q=k-1}^n\|\nabla\bar{e}_c^{q}\|_1^2+C\tau\sum_{q=0}^n\|\nabla e_\rho^q\|_0^2.\label{eq.4.9}
		\end{align}
	\end{lemma}
	\begin{proof}
		For $q=k-1,\cdots,n$, taking the inner product of \eqref{eq.4.3} at $t^{q+1}$ with $\tau(\bar{e}_c^{q+1}-\sigma_k\bar{e}_c^q)$, and applying Green's theorem, we obtain
		\begin{align*}
			&\big(D_{k\tau}\bar{e}_c^{q+1},\tau(\bar{e}_c^{q+1}-\sigma_k\bar{e}_c^q)\big)+\big(\nabla \bar{e}_c^{q+1},\tau\nabla(\bar{e}_c^{q+1}-\sigma_k\bar{e}_c^q)\big)+\big(\bar{e}_c^{q+1},\tau(\bar{e}_c^{q+1}-\sigma_k\bar{e}_c^q)\big)\\ ={}&\big(B_k(e_\rho^q),\tau(\bar{e}_c^{q+1}-\sigma_k\bar{e}_c^q)\big)+\big(Q_k^{q+1}(\rho),\tau(\bar{e}_c^{q+1}-\sigma_k\bar{e}_c^q)\big)-\big(R_k^{q+1}(c),\tau(\bar{e}_c^{q+1}-\sigma_k\bar{e}_c^q)\big).
		\end{align*}
		By applying Lemma \ref{le.4.1}, \eqref{eq.4.6}, Hölder's inequality, and Young's inequality to the above equation, he following derivation is obtained:
		\begin{align*}
			&\sum_{i,j=1}^kg_{ij}(\bar{e}_c^{q+1+i-k},\bar{e}_c^{q+1+j-k})-\sum_{i,j=1}^kg_{ij}(\bar{e}_c^{q+i-k},\bar{e}_c^{q+j-k})\\
			&+\|\sum_{i=0}^k\delta_i\bar{e}_c^{q+1+i-k}\|_0^2+\tau\|\nabla \bar{e}_c^{q+1}\|_0^2+\tau\|\bar{e}_c^{q+1}\|_0^2\\
			={}&\tau\sigma_k(\nabla\bar{e}_c^{q+1},\nabla\bar{e}_c^q)+\tau\sigma_k(\bar{e}_c^{q+1},\bar{e}_c^q)+\tau\big(B_k(e_\rho^q),\bar{e}_c^{q+1}-\sigma_k\bar{e}_c^q\big)\\
			&+\tau\big(Q_k^{q+1}(\rho),\bar{e}_c^{q+1}-\sigma_k\bar{e}_c^q\big)-\tau\big(R_k^{q+1}(c),\bar{e}_c^{q+1}-\sigma_k\bar{e}_c^q\big)\\
			\le{}& \tau\|\nabla \bar{e}_c^{q+1}\|_0\|\nabla \bar{e}_c^q\|_0+\tau\|\bar{e}_c^{q+1}\|_0\|\bar{e}_c^q\|_0+\tau\|B_k(e_\rho^q)\|_0(\|\bar{e}_c^{q+1}\|_0+\|\bar{e}_c^q\|_0)\\
			&+\tau\|Q_k^{q+1}(\rho)\|_0(\|\bar{e}_c^{q+1}\|_0+\|\bar{e}_c^q\|_0)+\tau \|R_k^{q+1}(c)\|_0(\|\bar{e}_c^{q+1}\|_0+\|\bar{e}_c^q\|_0)\\
			\le{}&\frac\tau2\|\nabla\bar{e}_c^{q+1}\|_0^2+C\tau\|\nabla\bar{e}_c^q\|_0^2+
			\tau\|\bar{e}_c^{q+1}\|_0^2+C\tau\|\bar{e}_c^q\|_0^2+C\tau\|B_k(e_\rho^q)\|_0^2\\
			&+C\tau^{2k}\int_{t_{q+1-k}}^{t_{q+1}}\big(\|\partial_t^k\rho\|_0^2+\|\partial_t^{k+1}c\|_0^2\big)\diff t.
		\end{align*}
		Then, by summing the above inequality for $q$ ranging from $k-1$ to $n$, and using Assumption~\ref{as.2.1} along with \eqref{eq.2.4}, the following result is obtained:
		\begin{align}\label{eq.4.10}
			&\sum_{i,j=1}^kg_{ij}(\bar{e}_c^{n+1+i-k},\bar{e}_c^{n+1+j-k})+\sum_{q=k-1}^n\|\sum_{i=0}^k\delta_i\bar{e}_c^{q+1+i-k}\|_0^2+\tau\sum_{q=k-1}^n\|\nabla\bar{e}_c^{q+1}\|_0^2\notag\\
			\le{}&C\sum_{i,j=1}^k\|\bar{e}_c^{i-1}\|_0\|\bar{e}_c^{j-1}\|_0+C\tau\sum_{q=k-1}^n\|\bar{e}_c^{q}\|_1^2+C\tau\sum_{q=k-1}^n\|B_k(e_\rho^q)\|_0^2\notag\\
			&+C\tau^{2k}\int_{0}^{t_{n+1}}\big(\|\partial_t^k\rho\|_0^2+\|\partial_t^{k+1}c\|_0^2\big)\diff t\notag\\
			\le{}&C\tau^{2k}+C\tau\sum_{q=k-1}^n\|\bar{e}_c^{q}\|_1^2+C\tau\sum_{q=0}^n\|e_\rho^q\|_0^2.
		\end{align}
		Substituting \eqref{eq.4.7} into \eqref{eq.4.10} and dropping some unnecessary terms, we immediately obtain \eqref{eq.4.8}.
		
		\eqref{eq.4.9} can be easily derived in an analogous manner by taking the inner product of \eqref{eq.4.3} at $t^{q+1}$ with $-\tau\varDelta(\bar{e}_c^{q+1}-\sigma_k\bar{e}_c^q)$.
	\end{proof}
	\begin{lemma}\label{le.4.3}
		Under the Assumption \ref{as.2.1}, for any $n=k-1,\cdots,N-1$, we have
		\begin{align}
			\|e_u^{n+1}\|_0^2&+\tau\sum_{q=k-1}^n\|\nabla e_u^{q+1}\|_0^2\le C\tau^{2k}+C\tau\sum_{q=0}^n\|\nabla e_u^q\|_0^2+C\tau\sum_{q=0}^n\|\nabla e_c^q\|_0^2\notag\\
			&+C\tau\sum_{q=k-1}^n\|\nabla\bar{e}_c^{q+1}\|_0^2+C\tau\sum_{q=k-1}^n\big(1+\|\nabla B_k(u^q)\|_0^2\|\varDelta B_k(u^q)\|_0^2\notag\\
			&+\|\nabla B_k(c^q)\|_0^2\|\varDelta B_k(c^q)\|_0^2\big)\big(\|e_u^{q+1}\|_0^2+\|e_u^q\|_0^2\big),\label{eq.4.11}\\
			\|\nabla e_u^{n+1}\|_0^2&+\tau\sum_{q=k-1}^n \|\varDelta e_u^{q+1}\|_0^2\le C\tau^{2k}+C\tau\sum_{q=0}^n\|\nabla e_u^q\|_1^2+C\tau\sum_{q=0}^n \|\nabla e_c^q\|_1^2\notag\\
			&+C\tau\sum_{q=k-1}^n\|\varDelta\bar{e}_c^{q+1}\|_0^2+C\tau\sum_{q=k-1}^n\big(1+\|\nabla B_k(u^q)\|_0^2\|\varDelta B_k(u^q)\|_0^2\notag\\
			&+\|\nabla B_k(c^q)\|_0^2\|\varDelta B_k(c^q)\|_0^2\big)\|\nabla e_u^{q+1}\|_0^2.\label{eq.4.12}
		\end{align}
	\end{lemma}
	\begin{proof}
		For $q=k-1,\cdots,n$, taking the inner product of \eqref{eq.4.4} at $t^{q+1}$ with $\tau(e_u^{q+1}-\sigma_ke_u^q)$ and using Lemma \ref{le.4.1} leads to
		\begin{align}\label{eq.4.13}
			&\sum_{i,j=1}^kg_{ij}(e_u^{q+1+i-k},e_u^{q+1+j-k})-\sum_{i,j=1}^kg_{ij}(e_u^{q+i-k},e_u^{q+j-k})+\|\sum_{i=0}^k\delta_ie_u^{q+1+i-k}\|_0^2+\tau\|\nabla e_u^{q+1}\|_0^2\notag \\
			={}&b\big(u(t^{q+1}),Q_k^{q+1}(u)+B_k(e_u^q),\tau(e_u^{q+1}-\sigma_ke_u^q)\big)+b\big(e_u^{q+1},B_k(u^q),\tau(e_u^{q+1}-\sigma_ke_u^q)\big)\notag\\
			&-b\big(u(t^{q+1}),Q_k^{q+1}(c)+B_k(e_c^q),\tau(e_u^{q+1}-\sigma_ke_u^q)\big)-b\big(e_u^{q+1},B_k(c^q),\tau(e_u^{q+1}-\sigma_ke_u^q)\big)\notag\\
			&+a\big(\bar{e}_c^{q+1},\tau(e_u^{q+1}-\sigma_ke_u^q)\big)-\big(R_k^{q+1}(u),\tau(e_u^{q+1}-\sigma_ke_u^q)\big)+a\big(e_u^{q+1},\tau\sigma_ke_u^q\big).
		\end{align}
		Next, Hölder's and Young's inequalities are applied to estimate the terms on the right-hand side of \eqref{eq.4.13}. First, applying \eqref{eq.4.6a} to the final three linear terms, we have
		\begin{align*}
			&|a\big(\bar{e}_c^{q+1},\tau(e_u^{q+1}-\sigma_ke_u^q)\big)|+|a(e_u^{q+1},\tau\sigma_ke_u^q)|\\
			&\le\frac\tau4\|\nabla e_u^{q+1}\|_0^2+C\tau\|\nabla e_u^{q}\|_0^2+C\tau\|\nabla\bar{e}_c^{q+1}\|_0^2,
		\end{align*}
		and
		\begin{align*}
			|-\big(R_k^{q+1}(u),\tau(e_u^{q+1}-\sigma_ke_u^q)\big)|\le C\tau(\|e_u^{q+1}\|_0^2+\|e_u^q\|_0^2)+C\tau^{2k}\int_{t_{q+1-k}}^{t_{q+1}} \|\partial_t^{k+1}u\|_0^2 \diff t.
		\end{align*}
		Then, for the first and third nonlinear terms, using \eqref{eq.2.1b}, \eqref{eq.2.2a}, Assumption \ref{as.2.1} and \eqref{eq.4.6b}, we derive
		\begin{align*}
			&|b\big(u(t^{q+1}),Q_k^{q+1}(u)+B_k(e_u^q),\tau(e_u^{q+1}-\sigma_ke_u^q)\big)|\\
			&+|-b\big(u(t^{q+1}),Q_k^{q+1}(c)+B_k(e_c^q),\tau(e_u^{q+1}-\sigma_ke_u^q)\big)|\\
			\le{}&\tau\|\nabla u(t^{q+1})\|_{L^4}\big(\|\nabla\big(Q_k^{q+1}(u)+B_k(e_u^q)\big)\|_0\\
			&+\|\nabla\big(Q_k^{q+1}(c)+B_k(e_c^q)\big)\|_0\big)\|e_u^{q+1}-\sigma_ke_u^q\|_{L^4}\\
			\le{}&C\tau\big(\|\nabla Q_k^{q+1}(u)\|_0+\|\nabla B_k(e_u^q)\|_0+\|\nabla Q_k^{q+1}(c)\|_0\\
			&+\|\nabla B_k(e_c^q)\|_0\big)\|e_u^{q+1}-\sigma_ke_u^q\|_0^{\frac12}\|\nabla (e_u^{q+1}-\sigma_ke_u^q)\|_0^{\frac12}\\
			\le{}&\frac\tau4\|\nabla e_u^{q+1}\|_0^2+C\tau\|\nabla e_u^q\|_0^2+C\tau\|e_u^{q+1}\|_0^2+C\tau\|e_u^q\|_0^2\\
			&+C\tau\|\nabla B_k(e_u^q)\|_0^2+C\tau\|\nabla B_k(e_c^q)\|_0^2\\
			&+C\tau^{2k}\int_{t_{q+1-k}}^{t_{q+1}}\big(\|\nabla \partial_t^ku\|_0^2+\|\nabla\partial_t^kc\|_0^2\big)\diff t.
		\end{align*}
		Finally, combining \eqref{eq.2.1a} with \eqref{eq.2.2a}, the remaining two terms can be bounded by
		\begin{align*}
			&|b\big(e_u^{q+1},B_k(u^q),\tau(e_u^{q+1}-\sigma_ke_u^q)\big)|+ |-b\big(e_u^{q+1},B_k(c^q),\tau(e_u^{q+1}-\sigma_ke_u^q)\big)|\\
			\le{}& \tau\|\nabla e_u^{q+1}\|_0\big(\|\nabla B_k(u^q)\|_{L^4}+\|\nabla B_k(c^q)\|_{L^4}\big)\|e_u^{q+1}-\sigma_ke_u^q\|_{L^4}\\
			\le{}& \frac\tau4\|\nabla e_u^{q+1}\|_0^2+C\tau\|\nabla e_u^q\|_0^2+C\tau\big(\|\nabla B_k(u^q)\|_0^2 \|\varDelta B_k(u^q)\|_0^2\\
			&+\|\nabla B_k(c^q)\|_0^2\|\varDelta B_k(c^q)\|_0^2\big)\big(\|e_u^{q+1}\|_0^2+\|e_u^q\|_0^2\big).
		\end{align*}
		Putting above estimates into \eqref{eq.4.13}, taking the sum from $q=k-1$ to $n$, and using Assumption \ref{as.2.1} and \eqref{eq.4.7}, we derive \eqref{eq.4.11}.
		
		Next, we deduce \eqref{eq.4.12} in a similar manner as we did to conclude \eqref{eq.4.11}. Letting \eqref{eq.4.4} take the inner product with $-\tau\varDelta(e_u^{q+1}-\sigma_ke_u^q)$ and applying Lemma \ref{le.4.1}, the following result is obtained:
		\begin{align}\label{eq.4.14}
			&\sum_{i,j=1}^kg_{ij}(\nabla e_u^{q+1+i-k},\nabla e_u^{q+1+j-k})-\sum_{i,j=1}^kg_{ij}(\nabla e_u^{q+i-k},\nabla e_u^{q+j-k})\notag\\
			&+\|\sum_{i=0}^k\delta_i\nabla e_u^{q+1+i-k}\|_0^2+\tau\|\varDelta e_u^{q+1}\|_0^2\notag\\
			=&-b\big(u(t^{q+1}),Q_k^{q+1}(u)+B_k(e_u^q),\tau\varDelta(e_u^{q+1}-\sigma_ke_u^q)\big)\notag\\
			&-b\big(e_u^{q+1},B_k(u^q),\tau\varDelta(e_u^{q+1}-\sigma_ke_u^q)\big)\notag\\
			&+b\big(u(t^{q+1}),Q_k^{q+1}(c)+B_k(e_c^q),\tau\varDelta(e_u^{q+1}-\sigma_ke_u^q)\big)\notag\\
			&+b\big(e_u^{q+1},B_k(c^q),\tau\varDelta(e_u^{q+1}-\sigma_ke_u^q)\big)+\big(\varDelta\bar{e}_c^{q+1},\tau\varDelta(e_u^{q+1}-\sigma_ke_u^q)\big)\notag\\
			&+\big(R_k^{q+1}(u),\tau\varDelta(e_u^{q+1}-\sigma_ke_u^q)\big)+\big(\varDelta e_u^{q+1},\tau\sigma_k\varDelta e_u^q\big).
		\end{align}
		Using \eqref{eq.2.1c}, \eqref{eq.2.2a}, Assumption \ref{as.2.1} and \eqref{eq.4.6b}, we obtain the following bounds for the first four terms on the right-hand side of \eqref{eq.4.14}:
		\begin{align*}
			&|-b\big(u(t^{q+1}),Q_k^{q+1}(u)+B_k(e_u^q),\tau\varDelta(e_u^{q+1}-\sigma_ke_u^q)\big)|\\
			&+|b\big(u(t^{q+1}),Q_k^{q+1}(c)+B_k(e_c^q),\tau\varDelta(e_u^{q+1}-\sigma_ke_u^q)\big)|\\
			\le{}&\tau\|\nabla u(t^{q+1})\|_{L^4}\big(\|\nabla\big(Q_k^{q+1}(u)+B_k(e_u^q)\big)\|_{L_4}\\
			&+\|\nabla\big(Q_k^{q+1}(c)+B_k(e_c^q)\big)\|_{L_4}\big)\|\varDelta(e_u^{q+1}-\sigma_ke_u^q)\|_0\\
			\le{}&\frac\tau4\|\varDelta e_u^{q+1}\|_0^2+C\tau\|\varDelta e_u^q\|_0^2+C\tau\|\nabla B_k(e_u^q)\|_0^2\\
			&+C\tau\|\varDelta B_k(e_u^q)\|_0^2+C\tau\|\nabla B_k(e_c^q)\|_0^2+C\tau\|\varDelta B_k(e_c^q)\|_0^2\\
			&+C\tau^{2k}\int_{t_{q+1-k}}^{t_{q+1}}\big(\|\nabla \partial_t^ku\|_0^2+\|\nabla \partial_t^kc\|_0^2+\|\varDelta\partial_t^ku\|_0^2+\|\varDelta\partial_t^kc\|_0^2\big)\diff t,
		\end{align*}
		and
		\begin{align*}
			&|-b\big(e_u^{q+1},B_k(u^q),\tau\varDelta(e_u^{q+1}-\sigma_ke_u^q)\big)|+|b\big(e_u^{q+1},B_k(c^q),\tau\varDelta(e_u^{q+1}-\sigma_ke_u^q)\big)|\\
			\le{}&\tau\|\nabla e_u^{q+1}\|_{L^4}\big(\|\nabla B_k(u^q)\|_{L^4}+\|\nabla B_k(c^q)\|_{L^4}\big)\|\varDelta(e_u^{q+1}-\sigma_ke_u^q)\|_0\\
			\le{}&C\tau\|\nabla e_u^{q+1}\|_0^{\frac12} \|\varDelta e_u^{q+1}\|_0^{\frac12} \big(\|\nabla B_k(u^q)\|_0^{\frac12}\|\varDelta B_k(u^q)\|_0^{\frac12}\\
			&+\|\nabla B_k(c^q)\|_0^{\frac12}\|\varDelta B_k(c^q)\|_0^{\frac12}\big)\big(\|\varDelta e_u^{q+1}\|_0+\|\varDelta e_u^q\|_0\big)\\
			\le{}& \frac\tau4\|\varDelta e_u^{q+1}\|_0^2+C\tau\|\varDelta e_u^q\|_0^2+C\tau\big(\|\nabla B_k(u^q)\|_0^2\|\varDelta B_k(u^q)\|_0^2\\
			&+\|\nabla B_k(c^q)\|_0^2\|\varDelta B_k(c^q)\|_0^2\big)\|\nabla e_u^{q+1}\|_0^2.
		\end{align*}
		Utilizing Hölder's inequality and Young's inequality to the final three terms in \eqref{eq.4.14}, we have
		\begin{align*}
			&|\big(\varDelta\bar{e}_c^{q+1},\tau\varDelta(e_u^{q+1}-\sigma_ke_u^q)\big)|+|(\varDelta e_u^{q+1},\tau\sigma_k\varDelta e_u^q)|\\
			\le{}&\frac\tau8\|\varDelta e_u^{q+1}\|_0^2+C\tau\|\varDelta e_u^q\|_0^2+C\tau\|\varDelta\bar{e}_c^{q+1}\|_0^2,
		\end{align*}
		and
		\begin{align*}
			|\big(R_u^{q+1},\tau\varDelta(e_u^{q+1}-\sigma_ke_u^q)\big)|\le\frac\tau8\|\varDelta e_u^{q+1}\|_0^2+C\tau\|\varDelta e_u^q\|_0^2+C\tau^{2k}\int_{t_{q+1-k}}^{t_{q+1}}\|\partial_t^{k+1}u\|_0^2\diff t.
		\end{align*}
		Substituting the above inequalities into \eqref{eq.4.14}, summing from $q=k-1$ to $n$, and using Assumption \ref{as.2.1} along with \eqref{eq.4.7}, we derive \eqref{eq.4.12}. The proof is completed.
	\end{proof}	
	\begin{lemma}\label{le.4.4}
		Under  Assumption \ref{as.2.1}, for any $n=k-1,\cdots,N-1$, we have
		\begin{align*}
			\|\bar{e}_\rho^{n+1}\|_0\le{}& C\exp\big(\max\{\|u(t^{n+1})\|_\infty,\|u^{n+1}\|_\infty\}\big)\|e_u^{n+1}\|_0,\\
			\|\nabla\bar{e}_\rho^{n+1}\|_0\le{}&C\big(\|\nabla e_u^{n+1}\|_0+\exp(\max\{\|u(t^{n+1})\|_\infty,\|u^{n+1}\|_\infty\})\|e_u^{n+1}\|_{L^4}\|\nabla u^{n+1}\|_{L^4}\big),\\
			|1-\lambda^{n+1}|\le{}& C\big(|1-\lambda^{n+1}|\|\bar{e}_\rho^{n+1}\|_0+\|\bar{e}_\rho^{n+1}\|_0+\|e_\rho^n\|_0\big),\\
			\|e_\rho^{n+1}\|_0\le{}& C\big(|1-\lambda^{n+1}|+|1-\lambda^{n+1}|\|\bar{e}_\rho^{n+1}\|_0+\|\bar{e}_\rho^{n+1}\|_0\big),\\
			\|\nabla e_\rho^{n+1}\|_0\le{}& C\big(|1-\lambda^{n+1}|+|1-\lambda^{n+1}|\|\nabla\bar{e}_\rho^{n+1}\|_0+\|\nabla\bar{e}_\rho^{n+1}\|_0\big).
		\end{align*}
	\end{lemma}
	\begin{proof}
		The proof of Lemma \ref{le.4.4} adopts a similar approach to that used in the proof of Lemma 3.3 presented in \cite{bib34}, while taking into account the mass conservation property given by equation \eqref{eq.1.2}.
	\end{proof}
	
	\begin{lemma}\label{le.4.5}
		Under the Assumption \ref{as.2.1}, for any $n=k-1,\cdots,N-1$, we have
		\begin{align*}
			|1-\mu^{n+1}|\le{}&C\tau^k+ C|1-\mu^{n+1}|\big(\|\bar{e}_c^{n+1}\|_0+\|\nabla\bar{e}_c^{n+1}\|_0+\|\bar{e}_c^{n+1}\|_1^2\big)\\
			&+C\Big(\big(\|\bar{e}_c^{n+1}\|_0+\|e_u^{n+1}\|_0+\max\{1/\lambda^{n+1},1\}|1-\lambda^{n+1}|\big)\|\rho^{n+1}\|_0\notag\\
			&+\|e_\rho^{n+1}\|_0+\|\bar{e}_c^{n+1}\|_0+\|\nabla\bar{e}_c^{n+1}\|_0+\|\bar{e}_c^{n+1}\|_1^2\Big)\notag\\
			&+C\sum_{j=n+1-k}^{n}\Big(\big(\|\bar{e}_c^j\|_0+\|e_u^j\|_0+\max\{1/\lambda^j,1\}|1-\lambda^j|\big)\|\rho^j\|_0\notag\\
			&+\|e_\rho^j\|_0+\|e_c^j\|_0+\|\nabla e_c^j\|_0+\|e_c^j\|_1^2\Big)\notag\\
			&+C\tau\|\rho^{n+1}\|_{\infty}\big(\|\nabla e_u^{n+1}\|_0+\|\nabla\bar{e}_c^{n+1}\|_0+\|\nabla e_u^{n+1}\|_0^2+\|\nabla\bar{e}_c^{n+1}\|_0^2\big)\notag\\
			&+C\tau\big(\|\varDelta\bar{e}_c^{n+1}\|_0+\|B_k(e_\rho^n)\|_0+\|\varDelta\bar{e}_c^{n+1}\|_0^2+\|B_k(e_\rho^n)\|_0^2\big)\notag.
		\end{align*}
	\end{lemma}
	\begin{proof}
		Subtracting \eqref{eq.3.5} from \eqref{eq.1.3} at $t^{n+1}$, we derive the follwing equation:
		\begin{align*}
			&\frac{d}{dt}E\big(\rho(t_{n+1}),c(t_{n+1})\big)-D_{k\tau}E\big(\rho(t_{n+1}),c(t_{n+1})\big)\\
			&+D_{k\tau}E\big(\rho(t_{n+1}),c(t_{n+1})\big)-D_{k\tau}E^{n+1}\\
			&+\int_{\Omega}\rho(t_{n+1})|\nabla\big(\log\rho(t_{n+1})-c(t_{n+1})\big)|^2-\rho^{n+1}|\nabla(\log\rho^{n+1}-\bar{c}^{n+1})|^2\diff \mathbf{x}\\
			&+\int_{\Omega}\big(c_t(t_{n+1})\big)^2-\big(D_{k\tau}c(t_{n+1})\big)^2\diff \mathbf{x}\notag\\
			&+\int_{\Omega}\big(D_{k\tau}c(t_{n+1})\big)^2-(D_{k\tau}\bar{c}^{n+1})^2\diff\mathbf{x}=0.
		\end{align*}
		Noting that there holds the following relationship:
		\begin{align*}
			&\tau D_{k\tau}E\big(\rho(t_{n+1}),c(t_{n+1})\big)-\tau D_{k\tau}E^{n+1}\\
			={}&\alpha_k\big(E(\rho(t_{n+1}),c(t_{n+1}))-E^{n+1}\big)+A_k\big(E(\rho(t_n),c(t_n))-E^n\big)\\
			={}&\frac{\alpha_k}2\int_\Omega(\bar{c}^{n+1})^2+|\nabla \bar{c}^{n+1}|^2-(c^{n+1})^2-|\nabla c^{n+1}|^2\diff\mathbf{x}\\
			&+\alpha_k\big(E(\rho(t_{n+1}),c(t_{n+1}))-E(\rho^{n+1},\bar{c}^{n+1})\big)+A_k\big(E(\rho(t_n),c(t_n))-E^n\big)\\
			={}&\frac{\alpha_k}2(1-\mu^{n+1})\int_\Omega(c(t_{n+1}))^2+|\nabla c(t_{n+1})|^2\diff\mathbf{x}\\
			&+\frac{\alpha_k}2(\mu^{n+1}-1)\int_\Omega(c(t_{n+1}))^2-(\bar{c}^{n+1})^2+|\nabla c(t_{n+1})|^2-|\nabla \bar{c}^{n+1}|^2\diff\mathbf{x}\\
			&+\alpha_k\big(E(\rho(t_{n+1}),c(t_{n+1}))-E(\rho^{n+1},\bar{c}^{n+1})\big)+A_k\big(E(\rho(t_n),c(t_n))-E^n\big).
		\end{align*}
		By incorporating this relationship into the previous equation, we get
		\begin{align}\label{eq.4.15}
			{}&\frac{\alpha_k}2(\mu^{n+1}-1)\int_\Omega\big(c(t_{n+1})\big)^2+|\nabla c(t_{n+1})|^2\diff\mathbf{x}\notag\\
			={}&\frac{\alpha_k}2(\mu^{n+1}-1)\int_\Omega\big(c(t_{n+1})\big)^2-(\bar{c}^{n+1})^2+|\nabla c(t_{n+1})|^2-|\nabla \bar{c}^{n+1}|^2\diff\mathbf{x}\notag\\
			&+\alpha_k\big(E(\rho(t_{n+1}),c(t_{n+1}))-E(\rho^{n+1},\bar{c}^{n+1})\big)\notag\\
			&+A_k\big(E(\rho(t_n),c(t_n))-E^n\big)\notag\\
			&+\tau\int_{\Omega}\rho(t_{n+1}) |\nabla\big(\log \rho(t_{n+1}) -c(t_{n+1})\big)|^2-\rho^{n+1}|\nabla(\log \rho^{n+1}-\bar{c}^{n+1})|^2\diff\mathbf{x}\notag\\
			&+\tau\int_{\Omega}\big(c_t(t_{n+1})\big)^2-\big(D_{k\tau}c(t_{n+1})\big)^2\diff \mathbf{x}\notag\\
			&+\tau\int_{\Omega}\big(D_{k\tau}c(t_{n+1})\big)^2-(D_{k\tau}\bar{c}^{n+1})^2\diff \mathbf{x}\notag\\
			&+\tau\Big(\frac{d}{dt}E\big(\rho(t_{n+1}),c(t_{n+1})\big)-D_{k\tau}E\big(\rho(t_{n+1}),c(t_{n+1})\big)\Big)\notag\\
			=:&\sum_{i=1}^{7}I_i.
		\end{align}
		Then, we utilize Hölder's inequality to establish upper bounds for the terms $I_1,\dots,I_7$. By employing Assumption \ref{as.2.1} and the identity
		\begin{align}
			a^2-b^2=2a(a-b)-(a-b)^2,\label{eq.4.16}
		\end{align}
		we derive $I_1$ as follows:
		\begin{align*}
			|I_1|\le{}&C|1-\mu^{n+1}|\int_\Omega|c(t_{n+1})\bar{e}_c^{n+1}|+|\nabla c(t_{n+1})\cdot\nabla\bar{e}_c^{n+1}|\diff\mathbf{x}\\
			&+\int_\Omega(\bar{e}_c^{n+1})^2+|\nabla \bar{e}_c^{n+1}|^2\diff\mathbf{x}\\
			\le&C|1-\mu^{n+1}|\big(\|\bar{e}_c^{n+1}\|_0+\|\nabla\bar{e}_c^{n+1}\|_0+\|\bar{e}_c^{n+1}\|_1^2\big).
		\end{align*}
		Based on Assumption \ref{as.2.1}, \eqref{eq.2.3c} and \eqref{eq.4.16}, we can derive the following inequality for $I_2$:
		\begin{align*}
			|I_2|={}&\alpha_k\bigg|\int_\Omega\Big(\big(\rho(t_{n+1})\log\rho(t_{n+1})-\rho^{n+1}\log\rho^{n+1}\big)-\big(\rho(t_{n+1})-\rho^{n+1}\big)\\
			{}&-\big(\rho(t_{n+1})c(t_{n+1})-\rho^{n+1}\bar{c}^{n+1}\big)+\frac12\big((c(t_{n+1}))^2-(\bar{c}^{n+1})^2\big)\\
			{}&+\frac12\big(|\nabla c(t_{n+1})|^2-|\nabla\bar{c}^{n+1}|^2\big)\Big)\diff \mathbf{x}\bigg|\\
			\le{}&C\Big(\|e_{\rho}^{n+1}\|_0\|u(t_{n+1})\|_0+\|\log\rho(t_{n+1})-\log\rho^{n}\|_0\|\rho^{n+1}\|_0\\
			{}&+\|e_{\rho}^{n+1}\|_0+\|e_\rho^{n+1}\|_0\|c(t_{n+1})\|_0+\|\rho^{n+1}\|_0\|\bar{e}_c^{n+1}\|_0\\
			&+\|c(t_{n+1})\|_0\|\bar{e}_c^{n+1}\|_0+\|\bar{e}_c^{n+1}\|_0^2+\|\nabla c(t_{n+1})\|_0\|\nabla\bar{e}_c^{n+1}\|_0+\|\nabla\bar{e}_c^{n+1}\|_0^2\Big)\\
			\le{}&C\Big(\big(\|\bar{e}_c^{n+1}\|_0+\|e_u^{n+1}\|_0+\max\{1/\lambda^{n+1},1\}|1-\lambda^{n+1}|\big)\|\rho^{n+1}\|_0\\
			&+\|e_\rho^{n+1}\|_0+\|\bar{e}_c^{n+1}\|_0+\|\nabla\bar{e}_c^{n+1}\|_0+\|\bar{e}_c^{n+1}\|_1^2\Big),
		\end{align*}
		where the final step relies on the inequality given by
		\begin{align}\label{eq.4.17}
			\|\log\rho(t_{n+1})-\log\rho^{n+1}\|_0={}&\|\log\rho(t_{n+1})-\log(\lambda^{n+1}\bar{\rho}^{n+1})\|_0\notag\\
			={}&\|u(t_{n+1})-u^{n+1}+\log1-\log\lambda^{n+1}\|_0\notag\\
			\le&\|e_u^{n+1}\|_0+\max\{1/\lambda^{n+1},1\}|1-\lambda^{n+1}|,
		\end{align}
		which is obtained by using \eqref{eq.2.3c}, \eqref{eq.3.3}, \eqref{eq.3.4} and Lagrange's mean value theorem. Here, we note that $\lambda^{n+1}>0$. This fact is established by integrating both sides of \eqref{eq.3.3} and substituting \eqref{eq.3.4} to get
		\begin{align*}
			\int_{\Omega}\rho^{n+1}\diff\mathbf{x}=\int_{\Omega}\lambda^{n+1}\bar{\rho}^{n+1}\diff\mathbf{x}=\int_{\Omega}\rho^n\diff\mathbf{x}.
		\end{align*}
		Next, by substituting $n=k-1,\cdots,N-1$ into the above equality, we derive
		\begin{align}\label{mass conservation}
			\int_{\Omega}\rho^{n+1}\diff \mathbf{x}=\int_{\Omega}\rho^{k-1}\diff \mathbf{x}=\int_{\Omega}\rho_0\diff \mathbf{x}>0.
		\end{align}
		Subsequently, from \eqref{eq.3.4}, we conclude that $\bar{\rho}^{n+1}>0$ and, consequently, $\lambda^{n+1}>0$. This ensures that the term involving $\lambda^{n+1}$ in inequality \eqref{eq.4.17} is well-defined. Similar to the analysis for $I_2$, it holds for $I_3$ that
		\begin{align*}
			|I_3|={}&\Big|A_k\big(E\big(\rho(t_n),c(t_n)\big)-E^n\big)\Big|\\
			\le{}&C\sum_{j=n+1-k}^{n}\Big(\big(\|\bar{e}_c^j\|_0+\|e_u^j\|_0+\max\{1/\lambda^j,1\}|1-\lambda^j|\big)\|\rho^j\|_0\\
			&+\|e_\rho^j\|_0+\|e_c^j\|_0+\|\nabla e_c^j\|_0+\|e_c^j\|_1^2\Big).
		\end{align*}
		Combining Assumption \ref{as.2.1}, \eqref{eq.2.2a}, \eqref{eq.4.16}, \eqref{eq.4.17}, we arrive at
		\begin{align*}
			|I_4|\le{}&\tau\int_{\Omega}\Big(|e_\rho^{n+1}|\big|\nabla\big(\log\rho(t_{n+1})-c(t_{n+1})\big)\big|^2\\
			&+|\rho^{n+1}|\big||\nabla\big(\log\rho(t_{n+1}) -c(t_{n+1})\big)|^2-|\nabla(\log\rho^{n+1}-\bar{c}^{n+1})|^2\big|\Big)\diff \mathbf{x}\\
			\le{}&C\tau\Big(\|e_\rho^{n+1}\|_0\|\nabla\big(u(t_{n+1})-c(t_{n+1})\big)\|_{L^4}^2\\
			&+\|\rho^{n+1}\|_{\infty}\|\nabla\big(u(t_{n+1})-c(t_{n+1})\big)\|_0\big(\|\nabla e_u^{n+1}\|_0+\|\nabla\bar{e}_c^{n+1}\|_0\big)\\
			&+\|\rho^{n+1}\|_{\infty}\big(\|\nabla e_u^{n+1}\|_0^2+\|\nabla\bar{e}_c^{n+1}\|_0^2\big)\Big)\\
			\le{}&C\tau\|\rho^{n+1}\|_{\infty}\big(\|\nabla e_u^{n+1}\|_0+\|\nabla\bar{e}_c^{n+1}\|_0+\|\nabla e_u^{n+1}\|_0^2+\|\nabla\bar{e}_c^{n+1}\|_0^2\big)+\|e_\rho^{n+1}\|_0.
		\end{align*}
		In terms of Assumption \ref{as.2.1}, \eqref{eq.4.6a} and \eqref{eq.4.16}, we obtain
		\begin{align}
			|I_5|={}&\tau\bigg|\int_{\Omega}\big(c_t(t_{n+1})\big)^2-\big(D_{k\tau}c(t_{n+1})\big)^2\diff\mathbf{x}\bigg|\notag\\
			\le{}&C\tau\big(\|c_t(t_{n+1})\|_0\|R_k^{n+1}(c)\|_0+\|R_k^{n+1}(c)\|_0^2\big)\le C\tau^k.\notag
		\end{align}
		For the term $I_6$, thanks to Assumption \ref{as.2.1}, \eqref{eq.4.3}, \eqref{eq.4.5}, \eqref{eq.4.6} and \eqref{eq.4.16}, we have
		\begin{align*}
			|I_6|={}&\tau\bigg|\int_{\Omega}\big(D_{k\tau}c(t_{n+1})\big)^2-(D_{k\tau}\bar{c}^{n+1})^2\diff \mathbf{x}\bigg|\\
			\le{}&C\tau\|D_{k\tau}c(t_{n+1})\|_0\|D_{k\tau}\bar{e}_c^{n+1}\|_0+C\tau\|D_{k\tau}\bar{e}_c^{n+1}\|_0^2\\
			\le{}&C\tau\big(\|c_t(t_{n+1})\|_0+\|R_k^{n+1}(c)\|_0\big)\cdot\\
			&\big(\|\varDelta\bar{e}_c^{n+1}\|_0+\|\bar{e}_c^{n+1}\|_0+\|B_k(e_\rho^n)\|_0+\|Q_k^{n+1}(\rho)\|_0+\|R_k^{n+1}(c)\|_0\big)\\
			&+C\tau\big(\|\varDelta\bar{e}_c^{n+1}\|_0+\|\bar{e}_c^{n+1}\|_0+\|B_k(e_\rho^n)\|_0+\|Q_k^{n+1}(\rho)\|_0+\|R_k^{n+1}(c)\|_0\big)^2\\
			\le&C\tau\big(\tau^{k-\frac12}+\|\varDelta\bar{e}_c^{n+1}\|_0+\|\bar{e}_c^{n+1}\|_0+\|B_k(e_\rho^n)\|_0\big)\\
			&+C\tau\big(\tau^{k-\frac12}+\|\varDelta\bar{e}_c^{n+1}\|_0+\|\bar{e}_c^{n+1}\|_0+\|B_k(e_\rho^n)\|_0\big)^2.
		\end{align*}
		The estimation of the last term $I_7$ is provided as:
		\begin{align}\label{eq.4.18}
			|I_7|\le{}&\tau\bigg|\int_\Omega(\rho\log\rho)_t(t_{n+1})-D_{k\tau}\big(\rho(t_{n+1})\log\rho(t_{n+1})\big)\diff\mathbf{x}\bigg|\notag\\
			&+\tau\bigg|\int_\Omega\rho_t(t_{n+1})-D_{k\tau}\rho(t_{n+1})\diff\mathbf{x}\bigg|\notag\\
			&+\tau\bigg|\int_\Omega(\rho c)_t(t_{n+1})-D_{k\tau}\big(\rho(t_{n+1})c(t_{n+1})\big)\diff\mathbf{x}\bigg|\notag\\ &+\frac\tau2\bigg|\int_\Omega\big((c)^2\big)_t(t_{n+1})-D_{k\tau}\big(c(t_{n+1})\big)^2\diff\mathbf{x}\bigg|\notag\\
			&+\frac\tau2\bigg|\int_\Omega\big(|\nabla c|^2\big)_t(t_{n+1})-D_{k\tau}|\nabla c(t_{n+1})|^2\diff\mathbf{x}\bigg|\\
			=&:\sum_{i=1}^{5}I_{7i}.
		\end{align}
		Applying Assumption~\ref{as.2.1} and \eqref{eq.4.6a} to the second term $I_{72}$, we have
		\begin{align*}
			I_{72}={}&\tau\bigg|\int_\Omega\rho_t(t_{n+1})-D_{k\tau}\rho(t_{n+1})\diff\mathbf{x}\bigg|\le C\|R_k^{n+1}(\rho)\|_0\\
			\le{}&C\tau^{k+\frac12}\bigg(\int_{t_{n+1-k}}^{t_{n+1}}\|\partial_t^{k+1}\rho\|_0^2\diff t\bigg)^{\frac12}\le C\tau^k.
		\end{align*}
		For the first term $I_{71}$, by Assumption~\ref{as.2.1}, \eqref{eq.2.2a}, \eqref{eq.2.4}, and \eqref{eq.4.5}, we get
		\begin{align*}
			I_{71}={}&\tau\bigg|\int_\Omega(\rho\log\rho)_t(t_{n+1})-D_{k\tau}\big(\rho(t_{n+1})\log\rho(t_{n+1})\big)\diff\mathbf{x}\bigg|\\
			={}&\bigg|\int_\Omega\sum_{j=1}^k\Big(r_j\int_{t_{n+1-j}}^{t_{n+1}} (t^{n+1-j}-t)^k \partial_t^{k+1}(\rho\log\rho)\diff t\Big)\diff\mathbf{x}\bigg|\\
			\le{}&C\tau^{k+\frac12}\bigg(\int_{t_{n+1-k}}^{t_{n+1}}\| \partial_t^{k+1}(\rho u)\|_0^2\diff t\bigg)^\frac12\\
			\le{}& C\tau^{k+\frac12}\bigg(\sum_{m=0}^{k+1}\int_{t_{n+1-k}}^{t_{n+1}} \|\partial_t^m\rho \cdot\partial_t^{k+1-m}u\|_0^2\diff t\bigg)^\frac12\\
			\le{}& C\tau^{k+\frac12}\bigg(\sum_{m=0}^{k+1}\int_{t_{n+1-k}}^{t_{n+1}} \|\partial_t^m\rho\|_{L^4}^4+\|\partial_t^mu\|_{L^4}^4\diff t\bigg)^\frac12\\
			\le{}& C\tau^{k+\frac12}\bigg(\sum_{m=0}^{k+1}\int_{t_{n+1-k}}^{t_{n+1}}  \|\partial_t^m\rho\|_0^4+\|\nabla\partial_t^m\rho\|_0^4+\|\partial_t^mu\|_0^4+\|\nabla\partial_t^mu\|_0^4\diff t\bigg)^\frac12\\
			\le{}& C\tau^k.
		\end{align*}
		Similar to the above estimate, the remaining terms in \eqref{eq.4.18} can be bounded in the same way. Therefore, we have
		\begin{align*}
			|I_{73}|\le{}&C\tau^{k+\frac12}\bigg(\sum_{m=0}^{k+1}\int_{t_{n+1-k}}^{t_{n+1}}  \|\partial_t^m\rho\|_0^4+\|\nabla\partial_t^m\rho\|_0^4+\|\partial_t^mc\|_0^4+\|\nabla\partial_t^mc\|_0^4\diff t\bigg)^\frac12\\
			\le{}& C\tau^k,\\
			|I_{74}|\le{}&C\tau^{k+\frac12}\bigg(\sum_{m=0}^{k+1}\int_{t_{n+1-k}}^{t_{n+1}}  \|\partial_t^mc\|_0^4+\|\nabla\partial_t^mc\|_0^4\diff t\bigg)^\frac12\le C\tau^k,\\
			|I_{75}|\le{}&C\tau^{k+\frac12}\bigg(\sum_{m=0}^{k+1}\int_{t_{n+1-k}}^{t_{n+1}}  \|\nabla\partial_t^mc\|_0^4+\|\varDelta\partial_t^mc\|_0^4\diff t\bigg)^\frac12\le C\tau^k.
		\end{align*}
		By combining $I_{71},\dots,I_{75}$, we obtain
		\begin{align*}
			|I_{7}|\le C\tau^k.
		\end{align*}
		Finally, due to the fact that the system \eqref{eq.1.1} satisfies the asymptotic property~(see \cite{bib37}):
		\begin{equation*}
			\int_{\Omega}c(\mathbf{x},t)\diff\mathbf{x}=e^{-\frac\alpha\varepsilon t}\bigg(\int_{\Omega}c(\mathbf{x},0)-\frac1\alpha\rho(\mathbf{x},0)\diff \mathbf{x}\bigg)+\frac1\alpha\int_{\Omega}\rho(\mathbf{x},0)\diff\mathbf{x},\label{eq.1.5}
		\end{equation*}
		we can derive
		\begin{equation*}
			\int_{\Omega}c(\mathbf{x},t)\diff \mathbf{x}\ge\min\left\{\frac1\alpha\int_{\Omega}\rho(\mathbf{x},0)\diff\mathbf{x},\int_{\Omega}c(\mathbf{x},0)\diff \mathbf{x}\right\}>0,
		\end{equation*}
		by noting that $c_0(\mathbf{x})$ and $\rho_0(\mathbf{x})>0$. Applying the above inequality to the term on the left-hand side of \eqref{eq.4.15}, we derive
		\begin{align*}
			\int_\Omega\big(c(t_{n+1})\big)^2+|\nabla c(t_{n+1})|^2\diff\mathbf{x}\ge &\frac{1}{|\Omega|}\bigg(\int_\Omega c(t_{n+1})\diff\mathbf{x}\bigg)^2>0,
		\end{align*}
		where $|\Omega|$ represents the measure of the domain $\Omega$. By substituting the above inequality and the estimation results of $I_1,\dots,I_7$ into \eqref{eq.4.15}, Lemma \ref{le.4.5} can be proved.
	\end{proof}
	\begin{lemma}\label{le.4.6}
		Under the Assumption \ref{as.2.1}, for any $n=k-1,\cdots,N-1$, we have
		\begin{align*}
			\|e_c^{n+1}\|_0\le{}& C\big(|1-\sqrt{\mu}^{n+1}|(1+\|\bar{e}_c^{n+1}\|_0)+\|\bar{e}_c^{n+1}\|_0\big),\\
			\|\nabla e_c^{n+1}\|_0\le{}& C\big(|1-\sqrt{\mu}^{n+1}|(1+\|\nabla\bar{e}_c^{n+1}\|_0)+\|\nabla\bar{e}_c^{n+1}\|_0\big),\\
			\|\varDelta e_c^{n+1}\|_0\le{}& C\big(|1-\sqrt{\mu}^{n+1}|(1+\|\varDelta\bar{e}_c^{n+1}\|_0)+\|\varDelta\bar{e}_c^{n+1}\|_0\big).
		\end{align*}
	\end{lemma}
	\begin{proof}
		Recalling \eqref{eq.3.6}, we obtain
		\begin{align*}
			e_c^{n+1}={}&c(t_{n+1})-c^{n+1}=c(t_{n+1})-\sqrt{\mu}^{n+1}\bar{c}^{n+1}\\
			={}&(1-\sqrt{\mu}^{n+1})c(t_{n+1})+(\sqrt{\mu}^{n+1}-1)\bar{e}_c^{n+1}+\bar{e}_c^{n+1},
		\end{align*}
		which implies Lemma \ref{le.4.6}.
	\end{proof}
	Now, based on the results deduced in the above, we can easily demonstrate Theorem \ref{th.3.1}.
	\begin{proof}[Proof of Theorem \ref{th.3.1}]
		We will use mathematical induction to carry out the proof.
		
		(\Rmnum{1}). case of $n=k-1$:
		
		(\Rmnum{1}-1). Recalling Lemma \ref{le.4.2} and Assumption \ref{as.3.1}, we get
		\begin{subequations}\label{eq.4.19}
			\begin{align}		
				\|\bar{e}_c^k\|_0^2+\tau\|\nabla\bar{e}_c^k\|_0^2\le{}& C\tau^{2k}+C\tau\|\bar{e}_c^{k-1}\|_1^2+C\tau\sum_{q=0}^{k-1}\|e_\rho^q\|_0^2\le C\tau^{2k},\label{eq.4.19a}\\
				\|\nabla\bar{e}_c^k\|_0^2+\tau\|\varDelta\bar{e}_c^k\|_0^2\le{}& C\tau^{2k}+C\tau\|\nabla\bar{e}_c^{k-1}\|_1^2+C\tau\sum_{q=0}^{k-1}\|\nabla e_\rho^q\|_0^2\le C\tau^{2k},\label{eq.4.19b}
			\end{align}
		\end{subequations}
		which combining with the triangle inequality and Assumption \ref{as.2.1} yields
		\begin{align}
			\|\bar{c}^k\|_2^2\le\|c(t_k)\|_2^2+\|\bar{e}_c^k\|_2^2\le C.\label{eq.4.20}
		\end{align}
		
		(\Rmnum{1}-2). Using Lemma \ref{le.4.3}, \eqref{eq.4.19} and Assumption \ref{as.3.1}, we get
		\begin{align*}
			\|e_u^k\|_0^2+\tau\|\nabla e_u^k\|_0^2
			\le{}&C\tau^{2k}+C\tau\sum_{q=0}^{k-1}\|\nabla e_u^q\|_0^2+C\tau\sum_{q=0}^{k-1}\|\nabla e_c^q\|_0^2+C\tau\|\nabla\bar{e}_c^k\|_0^2\\
			{}&+C\tau\big(1+\|\nabla B_k(u^{k-1})\|_0^2\|\varDelta B_k(u^{k-1})\|_0^2\\
			&+\|\nabla B_k(c^{k-1})\|_0^2\|\varDelta B_k(c^{k-1})\|_0^2\big)\big(\|e_u^k\|_0^2+\|e_u^{k-1}\|_0^2\big)\\
			\le{}& C\tau^{2k}+C_1\tau\|e_u^k\|_0^2,\\
			\|\nabla e_u^k\|_0^2+\tau\|\varDelta e_u^k\|_0^2\le{}& C\tau^{2k}+C_2\tau\|\nabla e_u^k\|_0^2.
		\end{align*}	
		There exists a sufficiently small, positive constant $\tau_1$ for which $1-C_i\tau\ge\frac12$ $(i=1,2)$ holds if $\tau\le\tau_1$. When this is combined with the above inequalities, it implies
		\begin{subequations}\label{eq.4.21}
			\begin{align}
				\|e_u^k\|_0^2+\tau\|\nabla e_u^k\|_0^2\le C\tau^{2k},\label{eq.4.21a}\\
				\|\nabla e_u^k\|_0^2+\tau\|\varDelta e_u^k\|_0^2\le C\tau^{2k}.\label{eq.4.21b}
			\end{align}
		\end{subequations}
		Similar to \eqref{eq.4.20}, it is valid that
		\begin{align}
			\|u^k\|_2^2\le \|u(t_k)\|_2^2+\|e_u^k\|_2^2\le C.\label{eq.4.22}
		\end{align}
		
		(\Rmnum{1}-3). Based on Lemma \ref{le.4.4}, \eqref{eq.2.2}, \eqref{eq.4.21} and \eqref{eq.4.22}, Assumption~\ref{as.2.1}, we arrive at
		\begin{align}
			\|\bar{e}_\rho^k\|_0^2+\|\nabla\bar{e}_\rho^k\|_0^2+\|e_\rho^k\|_0^2+\|\nabla e_\rho^k\|_0^2+|1-\lambda^k|^2\le C\tau^{2k},\label{eq.4.23}
		\end{align}
		when $\tau\le\tau_2$ for some small positive constant $\tau_2$. Using the triangle inequality, we further obtain $\|\rho^k\|_1^2\le C$. It follows from \eqref{eq.2.2}, \eqref{eq.3.3}, \eqref{eq.3.4}, \eqref{eq.4.22} and \eqref{eq.4.23} that
		\begin{align*}
			\|\varDelta\rho^k\|_0^2={}&\|\varDelta(\lambda^k\bar{\rho}^k)\|_0^2=\|\lambda^k\exp(u^k)(\nabla u^k\cdot\nabla u^k+\varDelta u^k)\|_0^2\\
			\le{}& C\exp(2\|u^k\|_{\infty})(\|\nabla u^k\|_{L^4}^4+\|\varDelta u^k\|_0^2)\le C.
		\end{align*}
		Combining the above two inequalities yields
		\begin{align}
			\|\rho^k\|_2^2\le C.\label{eq.4.24}
		\end{align}
		
		(\Rmnum{1}-4). Applying Lemma \ref{le.4.5}, \eqref{eq.2.2}, \eqref{eq.4.19}, \eqref{eq.4.21a}, \eqref{eq.4.23}, \eqref{eq.4.24} and Assumption~\ref{as.3.1}, we have
		\begin{align}
			|1-\mu^k|\le C\tau^k+C_3\tau^k|1-\mu^k|.\notag
		\end{align}
		There exists a sufficiently small, positive constant $\tau_3$ for which $1-C_3\tau\ge\frac12$ holds whenever $\tau\le\tau_3$. This implies
		\begin{align*}
			|1-\mu^k|\le C\tau^k.
		\end{align*}
		Consequently, we can deduce from this inequality that when $\tau\le\tau_4$, where $\tau_4$ is a small positive constant such that $\mu^{k}\ge\frac12$ holds, we have
		\begin{align}
			|1-\sqrt{\mu^k}|\le C|1-\mu^k| \le C\tau^k.\label{eq.4.25}
		\end{align}
		
		(\Rmnum{1}-5). By combining Lemma \ref{le.4.6} with \eqref{eq.4.19} and \eqref{eq.4.25}, we arrive at
		\begin{align*}
			\|e_c^k\|_0^2+\|\nabla e_c^k\|_0^2+\tau\|\varDelta e_c^k\|_0^2\le C\tau^{2k}.
		\end{align*}
		Similar to the derivation of \eqref{eq.4.20}, we have
		\begin{align}
			\|c^k\|_2^2\le C.\notag
		\end{align}
		Therefore, \eqref{eq.3.8}-\eqref{eq.3.14} are valid when $n=k-1$.
		
		(\Rmnum{2}). Assuming that \eqref{eq.3.8}-\eqref{eq.3.14} hold true for $k-1\le n\le m-1$, where $k\le m\le N-1$, we will prove that \eqref{eq.3.8}-\eqref{eq.3.14} are also valid for $n = m$.
		
		(\Rmnum{2}-1). Using Lemma \ref{le.4.2} and the above assumption, we have
		\begin{align*}
			\|\bar{e}_c^{m+1}\|_0^2+\tau\sum_{q=k-1}^m\|\nabla\bar{e}_c^{q+1}\|_0^2
			\le {}&C\tau^{2k},\\
			\|\nabla\bar{e}_c^{m+1}\|_0^2+\tau\sum_{q=k-1}^m\|\varDelta\bar{e}_c^{q+1}\|_0^2\le{}& C\tau^{2k},
		\end{align*}
		and
		\begin{equation}		
			\|\bar{c}^{m+1}\|_2^2\le\|c(t_{m+1})\|_2^2+\|\bar{e}_c^{m+1}\|_2^2\le C.\notag
		\end{equation}
		
		(\Rmnum{2}-2). Using Lemma \ref{le.4.3}, the above inequalities, and the given assumption, we have
		\begin{align*}				
			\|e_u^{m+1}\|_0^2+\tau\sum_{q=k-1}^m\|\nabla e_u^{q+1}\|_0^2\le{}& C\tau^{2k}+C\tau\sum_{q=0}^m \|\nabla e_u^q\|_0^2\\
			&+C\tau\sum_{q=0}^m \|e_c^q\|_0^2+C\tau\sum_{q=k-1}^m \|\nabla\bar{e}_c^{q+1}\|_0^2\\
			&+C\tau\sum_{q=k-1}^m\big(1+\|\nabla B_k(u^q)\|_0^2\|\varDelta B_k(u^q)\|_0^2\\
			&+\|\nabla B_k(c^q)\|_0^2\|\varDelta B_k(c^q)\|_0^2\big)\big(\|e_u^{q+1}\|_0^2+\|e_u^q\|_0^2\big)\\
			\le{}& C\tau^{2k}+C\tau\sum_{q=k-1}^m\|e_u^{q+1}\|_0^2.
		\end{align*}
		Using the Gronwall lemma to the inequality above, the following result is obtained:
		\begin{align}				
			\|e_u^{m+1}\|_0^2+\tau\sum_{q=k-1}^m\|\nabla e_u^{q+1}\|_0^2\le C\tau^{2k}.\notag
		\end{align}
		Similarly, there holds
		\begin{align*}	
			\|\nabla e_u^{m+1}\|_0^2+\tau\sum_{q=k-1}^m\|\varDelta e_u^{q+1}\|_0^2\le C\tau^{2k}.
		\end{align*}
		Then we can get
		\begin{align*}	
			\|u^{m+1}\|_2^2\le \|u(t_{m+1})\|_2^2+\|e_u^{m+1}\|_2^2\le C.
		\end{align*}
		
		(\Rmnum{2}-3). Lemma \ref{le.4.4}, \eqref{eq.2.2}, Assumption \ref{as.2.1} and the above inequalities indicate that
		\begin{align*}	
			\|\bar{e}_\rho^{m+1}\|_0^2+\|\nabla\bar{e}_\rho^{m+1}\|_0^2+\|e_\rho^{m+1}\|_0^2+\|\nabla e_\rho^{m+1}\|_0^2+|1-\lambda^{m+1}|^2\le C\tau^{2k},
		\end{align*}
		when $\tau\le\tau_5$ for some small positive constant $\tau_5$.
		
		(\Rmnum{2}-4). Applying Lemma \ref{le.4.5}, \eqref{eq.2.2}, the above estimates, and the given assumption, we get
		\begin{align*}
			|1-\mu^{m+1}|\le C\tau^k,
		\end{align*}
		when $\tau\le\tau_6$ for some small positive constant $\tau_6$. Therefore, there exists another small positive constant $\tau_7$ such that $\mu^{m+1}\ge\frac12$ whenever $\tau\le\tau_7$. Combining this with the above inequality, we deduce that
		\begin{align*}
			|1-\sqrt{\mu^{m+1}}|^2\le{}&C\tau^{2k}.
		\end{align*}
		
		(\Rmnum{2}-5). Utilizing Lemma \ref{le.4.6} and the above inequalities, we can easily obtain
		\begin{align*}
			\|e_c^{m+1}\|_0^2+\|\nabla e_c^{m+1}\|_0^2+\tau\|\varDelta e_c^{m+1}\|_0^2\le& C\tau^{2k} \quad\text{and}\quad \|c^{m+1}\|_2^2\le C.
		\end{align*}
		
		The inequalities derived above demonstrate that \eqref{eq.3.8}-\eqref{eq.3.14} are also valid when $n=m$. In conclusion, we set $\tau^*=\min\{\tau_1,\cdots,\tau_7\}$ and combine the results from (\Rmnum{1}) and (\Rmnum{2}) to derive Theorem \ref{th.3.1} for all $\tau\le\tau^*$.
	\end{proof}
	\section{Numerical Experiments}\label{sec.5}
	This section provides numerical examples to verify the convergence rate of the constructed schemes \eqref{eq.3.1}-\eqref{eq.3.6} and confirm their properties of mass conservation, positivity preservation, and original energy-law preservation. For all computations below, the spatial domain is discretized using the finite element method with $P1$ element. Unless otherwise specified in the subsequent text, the parameters are set to $\alpha=\beta=\gamma=\varepsilon=1$.

	\textbf{Example 1.} We first verify the accuracy of these numerical schemes \eqref{eq.3.1}-\eqref{eq.3.6}. Consider the model \eqref{eq.1.1} with $\Omega=(-0.05,0.05)^2$ in the presence of external forcing, where the exact solutions are provided as
	\begin{align}
		c(\mathbf{x},t)={}&\sin(10\pi x)\sin(10\pi y)\sin(t)+1.1,\notag\\
		\rho(\mathbf{x},t)={}&\frac1{2\pi^2+1}\sin(10\pi x)\sin(10\pi y)\sin(t)+1.1.\notag
	\end{align}
	And the mesh size is given by $h=\tau^{k}$. Collecting numerical results at $T=1$ in Figure \ref{fig1} and Figure \ref{fig2}, we observe that all schemes achieve the predicted convergence order of $O(\tau^k)$ for all approximation solutions. These results are highly consistent with the theoretical predictions analysis in Theorem \ref{th.3.1}.
	
	\textbf{Example 2.} Then, we test the  physical
	properties preservation of the proposed scheme. Let $\Omega=(-5,5)^2$ and and consider the initial conditions given by
	\begin{align*}
		c(\mathbf{x},0)=e^{-\frac{x^2+y^2}2},\quad \rho(\mathbf{x},0)=4e^{-(x^2+y^2)},
	\end{align*}
	Taking $T=10$, $\alpha=6$, $\tau=0.001$, and $h=0.025$, we utilize the second-order scheme for temporal discretization. Figure \ref{fig3} displays the simulation results regarding the minimum values of $\rho$, cell mass, and energy. The observations indicate that the cell mass remains conserved, the positivity of $\rho$ is maintained, and the numerical energy exhibits dissipative behavior. Moreover, the numerical solutions presented in Figure \ref{fig4} align well with those in \cite{bib34}. These results align well with the previously derived theoretical predictions.
	
	\textbf{Example 3.} In the last example, we consider the case in which the solution will experience
	blow-up in finite-time. Let $\Omega=(-0.5, 0.5)^2$ with the following initial conditions \cite{bib25}:
	\begin{equation*}
		c(\mathbf{x}, 0)=420e^{-42(x^2+y^2)},\quad \rho(\mathbf{x}, 0)=840e^{-84(x^2+y^2)}.
	\end{equation*}
	It is obviously that  the initial mass of cells exceeds the critical value $8\pi$ (see \cite{bib12}) in this case. Setting $\tau=5\times10^{-6}$ and a spatial mesh size of $h=0.01$, we compare the proposed numerical scheme \eqref{eq.3.1}-\eqref{eq.3.6} ($k=2$) with the classical BDF$2$ scheme for the original formulation \eqref{eq.1.1}. Figure \ref{fig5} displays the one-dimensional profile of $\rho^N$ at $y = 0$ and $T=1\times10^{-4}$. It is evident from the figure that the cell density $\rho^N$ in the classical BDF$2$ scheme has negative values, whereas the proposed scheme maintains the solution strictly non-negative across the entire domain. The evolutions of cell mass and energy, computed by the new scheme, are presented in Figure \ref{fig6}, revealing that mass remains conserved while energy gradually decreases over time. Additionally, Figure \ref{fig7} presents the numerical solutions. These results show that the new proposed scheme still ensures the physical
	properties preservation under this limit state, which confirms the feasibility and efficiency of the proposed scheme again.
	
	\section{Conclusions}\label{sec.6}
	In this article, we first convert the Keller-Segel equation into its equivalent form through a logarithmic transformation. Then, by incorporating a recovery technique and an EPC method, we develop high-order temporal discrete schemes for solving the Keller-Segel equation using the BDF$k$ method. These schemes are not only mass conservative, positivity preserving, but also original energy-law preserving. Through a rigorous error analysis, the optimal convergence rate for all relevant functions under some realistic assumptions are derived in Theorem \ref{th.3.1}. Numerical examples have confirmed the correctness of the theoretical results and demonstrated that the solutions maintain the three properties mentioned above. However, this work focuses solely on the error analysis of the  temporal semi-discrete schemes and establishes their optimal convergence order. For the spatial discretization, we plan to employ the block-centered finite difference method described in~\cite{bib39}. A complete error analysis of the fully discrete scheme will be a focal point of our subsequent research. Furthermore, the Keller-Segel equation and more complex chemotaxis models possess a variety of additional properties. Designing numerical algorithms that can preserve as many physical properties of the continuous models as possible, while being applicable to a broader range of models, and deriving optimal error estimates for such schemes, constitutes a highly valuable research direction. This is not only one of the core objectives of this article but also a key focus for future research.
	\addcontentsline{toc}{section}{References}
	\bibliographystyle{plain}
	\bibliography{refs.bib}

\begin{thebibliography}{10}

\bibitem{bib30}
Daniel Acosta-Soba, Francisco Guill{\'e}n-Gonz{\'a}lez, and J~Rafael
  Rodr{\'\i}guez-Galv{\'a}n.
\newblock An unconditionally energy stable and positive upwind dg scheme for
  the keller-segel model.
\newblock {\em Journal of Scientific Computing}, 97(1):18, 2023.

\bibitem{bib14}
Gurusamy Arumugam and Jagmohan Tyagi.
\newblock Keller-segel chemotaxis models: A review.
\newblock {\em Acta Applicandae Mathematicae}, 171:1--82, 2021.

\bibitem{bib11}
Vincent Calvez and Lucilla Corrias.
\newblock The parabolic-parabolic keller-segel model in r2.
\newblock {\em Communications in Mathematical Sciences}, 6(2):417--447, 2008.

\bibitem{bib08}
Wenbin Chen, Qianqian Liu, and Jie Shen.
\newblock Error estimates and blow-up analysis of a finite-element
  approximation for the parabolic-elliptic keller-segel system.
\newblock {\em International Journal of Numerical Analysis and Modeling},
  19(2-3):275--298, 2022.

\bibitem{bib27}
Alina Chertock, Yekaterina Epshteyn, Hengrui Hu, and Alexander Kurganov.
\newblock High-order positivity-preserving hybrid
  finite-volume-finite-difference methods for chemotaxis systems.
\newblock {\em Advances in Computational Mathematics}, 44:327--350, 2018.

\bibitem{bib20}
Alina Chertock and Alexander Kurganov.
\newblock A second-order positivity preserving central-upwind scheme for
  chemotaxis and haptotaxis models.
\newblock {\em Numerische Mathematik}, 111(2):169--205, 2008.

\bibitem{bib10}
Lucilla Corrias, Miguel Escobedo, and Julia Matos.
\newblock Existence, uniqueness and asymptotic behavior of the solutions to the
  fully parabolic keller-segel system in the plane.
\newblock {\em Journal of Differential Equations}, 257(6):1840--1878, 2014.

\bibitem{bib29}
Mehdi Dehghan, Mostafa Abbaszadeh, and Akbar Mohebbi.
\newblock A meshless technique based on the local radial basis functions
  collocation method for solving parabolic-parabolic patlak-keller-segel
  chemotaxis model.
\newblock {\em Engineering Analysis with Boundary Elements}, 56:129--144, 2015.

\bibitem{bib38}
Jie Ding, Cheng Wang, and Shenggao Zhou.
\newblock A second-order accurate, original energy dissipative numerical scheme
  for chemotaxis and its convergence analysis.
\newblock {\em arXiv preprint}, 2024.

\bibitem{bib26}
Yekaterina Epshteyn and Alexander Kurganov.
\newblock New interior penalty discontinuous galerkin methods for the
  keller-segel chemotaxis model.
\newblock {\em SIAM Journal on Numerical Analysis}, 47(1):386--408, 2008.

\bibitem{bib37}
Xinlong Feng, Shuyan Ren, and Kun Wang.
\newblock Unconditionally positivity and structure preserving exponential time
  differencing schemes for the keller-segel equations.
\newblock Submitted, 2024.

\bibitem{bib17}
Francis Filbet.
\newblock A finite volume scheme for the patlak-keller-segel chemotaxis model.
\newblock {\em Numerische Mathematik}, 104:457--488, 2006.

\bibitem{bib09}
Herbert Gajewski and Klaus Zacharias.
\newblock Global behaviour of a reaction-diffusion system modelling chemotaxis.
\newblock {\em Mathematische Nachrichten}, 195:77--114, 1998.

\bibitem{bib19}
Li~Guo, Xingjie~Helen Li, and Yang Yang.
\newblock Energy dissipative local discontinuous galerkin methods for
  keller-segel chemotaxis model.
\newblock {\em Journal of Scientific Computing}, 78(3):1387--1404, 2019.

\bibitem{bib23}
Juan~Vicente Guti{\'e}rrez-Santacreu and Jos{\'e}~Rafael
  Rodr{\'\i}guez-Galv{\'a}n.
\newblock Analysis of a fully discrete approximation for the classical
  keller-segel model: Lower and a priori bounds.
\newblock {\em Computers \& Mathematics with Applications}, 85:69--81, 2021.

\bibitem{bib35}
Adrian~T. Hill and Endre S{\"u}li.
\newblock Approximation of the global attractor for the incompressible
  navier-stokes equations.
\newblock {\em IMA Journal of Numerical Analysis}, 20(4):633--667, 2000.

\bibitem{bib13}
Thomas Hillen and Kevin~J Painter.
\newblock A user's guide to pde models for chemotaxis.
\newblock {\em Journal of Mathematical Biology}, 58(1-2):183--217, 2009.

\bibitem{bib15}
Dirk Horstmann.
\newblock From 1970 until present: the keller-segel model in chemotaxis and its
  consequences.
\newblock {\em Jahresbericht der Deutschen Mathematiker-Vereinigung},
  105(3):103--165, 2003.

\bibitem{bib12}
Dirk Horstmann and Guofang Wang.
\newblock Blow-up in a chemotaxis model without symmetry assumptions.
\newblock {\em European Journal of Applied Mathematics}, 12(2):159--177, 2001.

\bibitem{bib05}
Jingwei Hu and Xiangxiong Zhang.
\newblock Positivity-preserving and energy-dissipative finite difference
  schemes for the fokker-planck and keller-segel equations.
\newblock {\em IMA Journal of Numerical Analysis}, 43(3):1450--1484, 2023.

\bibitem{bib32}
Fukeng Huang and Jie Shen.
\newblock Bound/positivity preserving and energy stable scalar auxiliary
  variable schemes for dissipative systems: Applications to keller-segel and
  poisson-nernst-planck equations.
\newblock {\em SIAM Journal on Scientific Computing}, 43(3):A1832--A1857, 2021.

\bibitem{bib36}
Fukeng Huang and Jie Shen.
\newblock A new class of implicit-explicit bdfk sav schemes for general
  dissipative systems and their error analysis.
\newblock {\em Computer Methods in Applied Mechanics and Engineering},
  392:114718, 2022.

\bibitem{bib02}
Evelyn~F Keller and Lee~A Segel.
\newblock Initiation of slime mold aggregation viewed as an instability.
\newblock {\em Journal of Theoretical Biology}, 26(3):399--415, 1970.

\bibitem{bib03}
Evelyn~F Keller and Lee~A Segel.
\newblock Model for chemotaxis.
\newblock {\em Journal of Theoretical Biology}, 30(2):225--234, 1971.

\bibitem{bib04}
Evelyn~F Keller and Lee~A Segel.
\newblock Traveling bands of chemotactic bacteria: a theoretical analysis.
\newblock {\em Journal of Theoretical Biology}, 30(2):235--248, 1971.

\bibitem{bib39}
Xiaoli Li, Jie Shen, and Hongxing Rui.
\newblock Energy stability and convergence of sav block-centered finite
  difference method for gradient flows.
\newblock {\em Mathematics of Computation}, 88(319):2047--2068, 2019.

\bibitem{bib24}
Xingjie~Helen Li, Chi-Wang Shu, and Yang Yang.
\newblock Local discontinuous galerkin method for the keller-segel chemotaxis
  model.
\newblock {\em Journal of Scientific Computing}, 73(2-3):943--967, 2017.

\bibitem{bib01}
Clifford~S Patlak.
\newblock Random walk with persistence and external bias.
\newblock {\em Bulletin of Mathematical Biophysics}, 15:311--338, 1953.

\bibitem{bib25}
Changxin Qiu, Qingyuan Liu, and Jue Yan.
\newblock Third order positivity-preserving direct discontinuous galerkin
  method with interface correction for chemotaxis keller-segel equations.
\newblock {\em Journal of Computational Physics}, 433:110191, 2021.

\bibitem{bib07}
Norikazu Saito.
\newblock Conservative upwind finite-element method for a simplified
  keller-segel system modelling chemotaxis.
\newblock {\em IMA journal of Numerical Analysis}, 27(2):332--365, 2007.

\bibitem{bib21}
Norikazu Saito.
\newblock Error analysis of a conservative finite-element approximation for the
  keller-segel system of chemotaxis.
\newblock {\em Communications on Pure and Applied Analysis}, 11(1):339--364,
  2012.

\bibitem{bib31}
Jie Shen and Jie Xu.
\newblock Unconditionally bound preserving and energy dissipative schemes for a
  class of keller-segel equations.
\newblock {\em SIAM Journal on Numerical Analysis}, 58(3):1674--1695, 2020.

\bibitem{bib22}
M~Sulman and T~Nguyen.
\newblock A positivity preserving moving mesh finite element method for the
  keller-segel chemotaxis model.
\newblock {\em Journal of Scientific Computing}, 80(1):649--666, 2019.

\bibitem{bib34}
Kun Wang, Enlong Liu, and Xinlong Feng.
\newblock Optimal error estimate of unconditionally positivity-preserving,
  mass-conserving and energy stable method for the keller-segel chemotaxis
  model.
\newblock {\em Mathematics of Computation}, 2024.

\bibitem{bib28}
Shufen Wang, Simin Zhou, Shuxun Shi, and Wenbin Chen.
\newblock Fully decoupled and energy stable bdf schemes for a class of
  keller-segel equations.
\newblock {\em Journal of Computational Physics}, 449:110799, 2022.

\bibitem{bib16}
Jie Xu and Hongfei Fu.
\newblock A decoupled linear, mass-conservative block-centered finite
  difference method for the keller-segel chemotaxis system.
\newblock {\em Journal of Computational Physics}, 526:113775, 2025.

\bibitem{bib33}
Wei Zheng and Yan Xu.
\newblock High-order decoupled and bound preserving local discontinuous
  galerkin methods for a class of chemotaxis models.
\newblock {\em Communications on Applied Mathematics and Computaion},
  6(1):372--398, 2024.

\bibitem{bib18}
Guanyu Zhou.
\newblock An analysis on the finite volume schemes and the discrete lyapunov
  inequalities for the chemotaxis system.
\newblock {\em Journal of Scientific Computing}, 87(2):1--47, 2021.

\bibitem{bib06}
Guanyu Zhou and Norikazu Saito.
\newblock Finite volume methods for a keller-segel system: discrete energy,
  error estimates and numerical blow-up analysis.
\newblock {\em Numerische Mathematik}, 135(1):265--311, 2017.

\end{thebibliography}
	
	\begin{figure}[htbp]
		\centering
		\begin{subfigure}{0.4\textwidth}
			\includegraphics[width=1\textwidth]{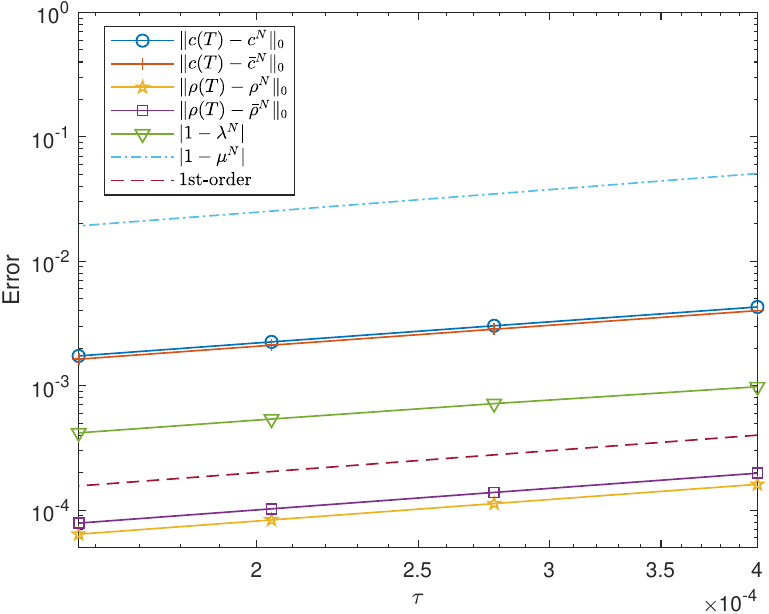}
			\caption{First-order scheme}
		\end{subfigure}
		\begin{subfigure}{0.4\textwidth}
			\includegraphics[width=1\textwidth]{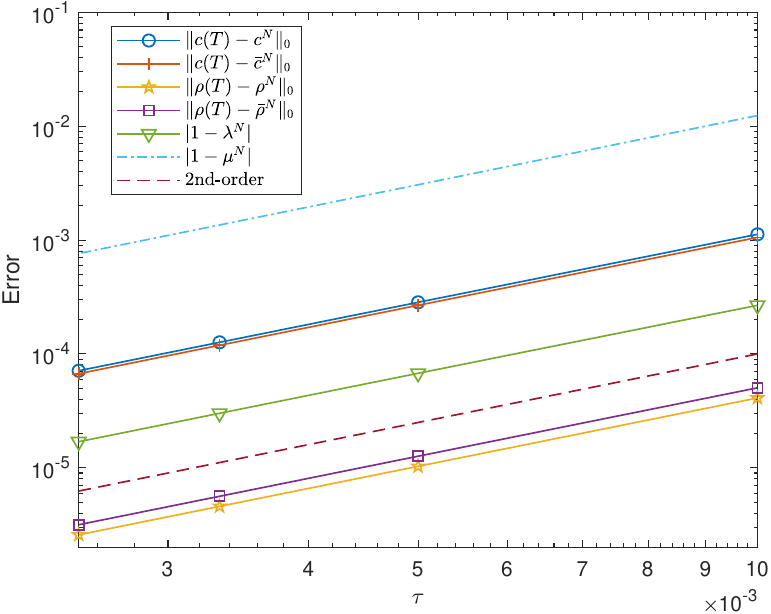}
			\caption{Second-order scheme}
		\end{subfigure}
		\begin{subfigure}{0.4\textwidth}
			\includegraphics[width=1\textwidth]{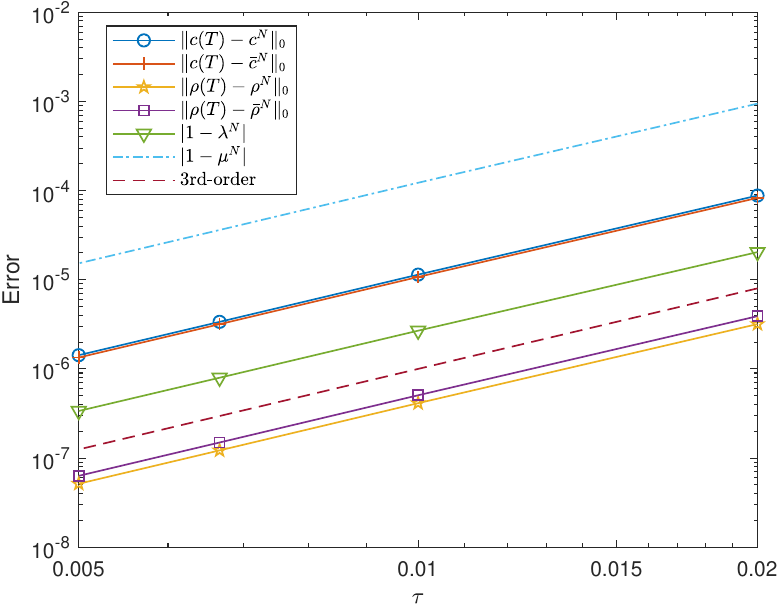}
			\caption{Third-order scheme}
		\end{subfigure}
		\begin{subfigure}{0.4\textwidth}
			\includegraphics[width=1\textwidth]{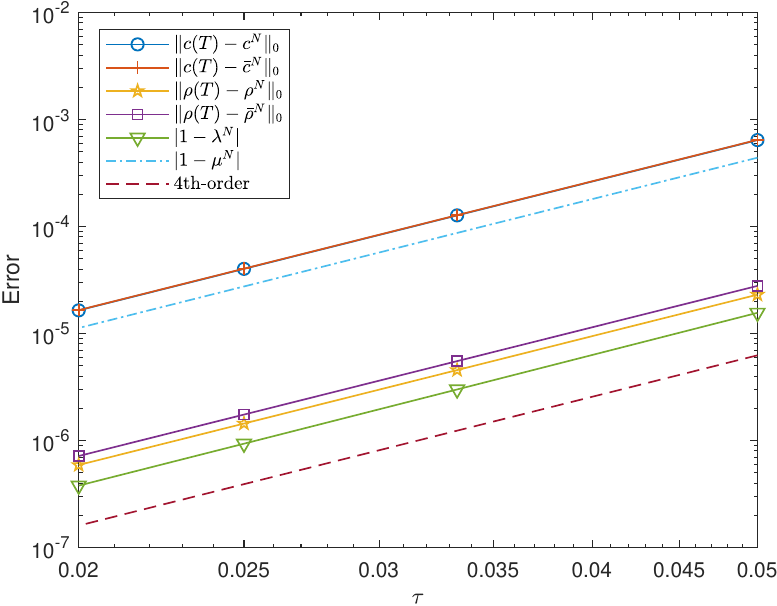}
			\caption{Fourth-order scheme}
		\end{subfigure}
		\begin{subfigure}{0.4\textwidth}
			\includegraphics[width=1\textwidth]{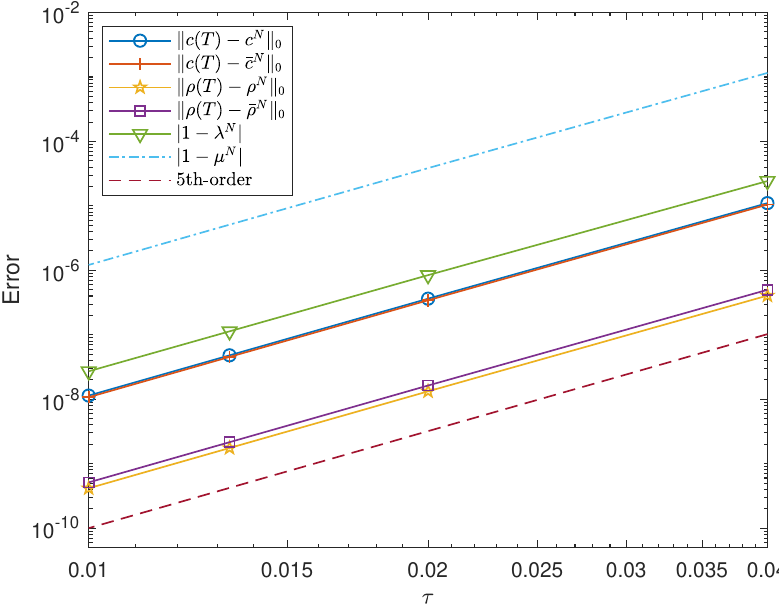}
			\caption{Fifth-order scheme}
		\end{subfigure}
		\caption{Convergence rates in  $L^2$-norm of the $k$th-order schemes.}\label{fig1}
	\end{figure}
	
	\begin{figure}
		\centering
		\begin{subfigure}{0.4\textwidth}
			\includegraphics[width=1\textwidth]{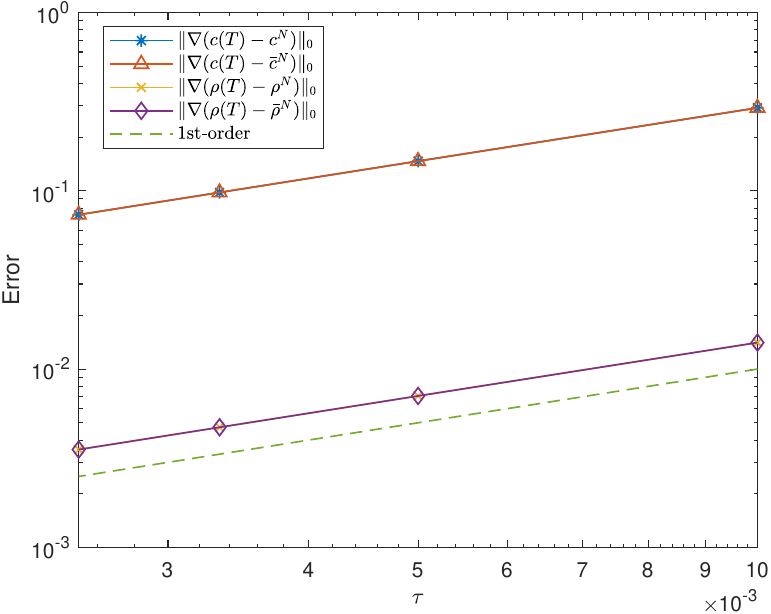}
			\caption{First-order scheme}
		\end{subfigure}
		\begin{subfigure}{0.4\textwidth}
			\includegraphics[width=1\textwidth]{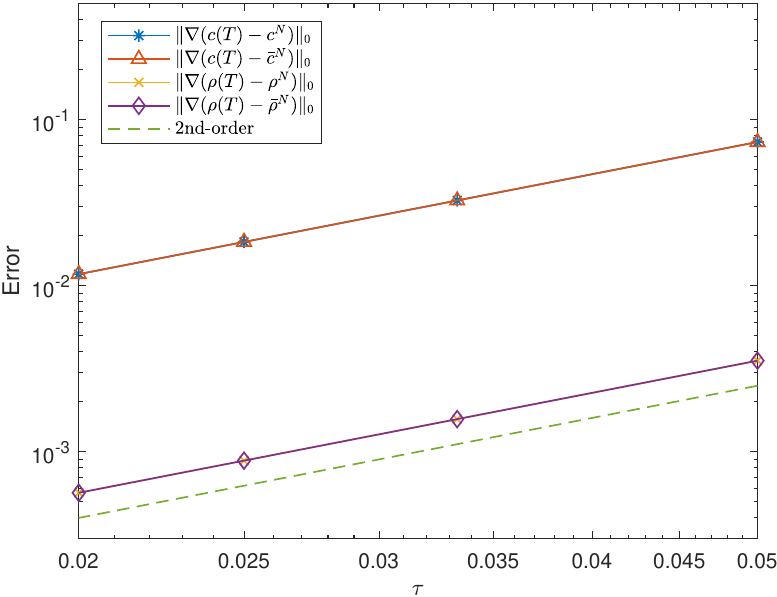}
			\caption{Second-order scheme}
		\end{subfigure}
		\begin{subfigure}{0.4\textwidth}
			\includegraphics[width=1\textwidth]{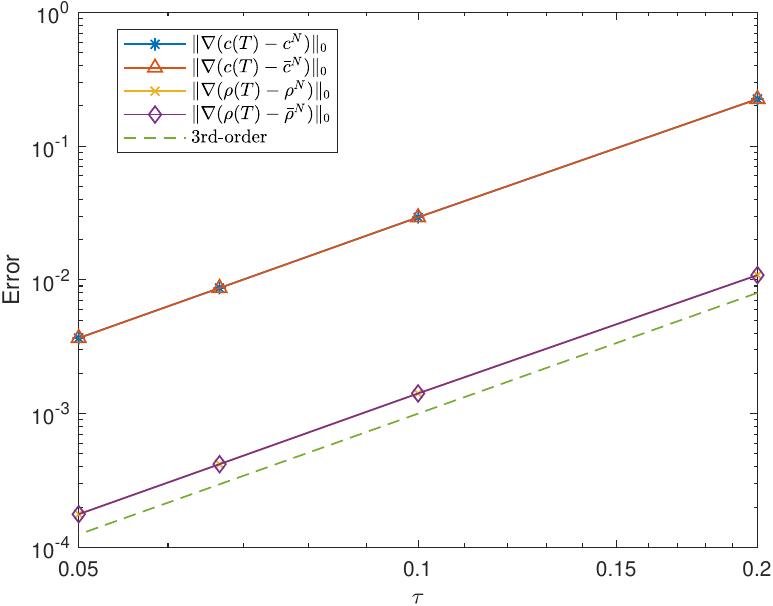}
			\caption{Third-order scheme}
		\end{subfigure}
		\begin{subfigure}{0.4\textwidth}
			\includegraphics[width=1\textwidth]{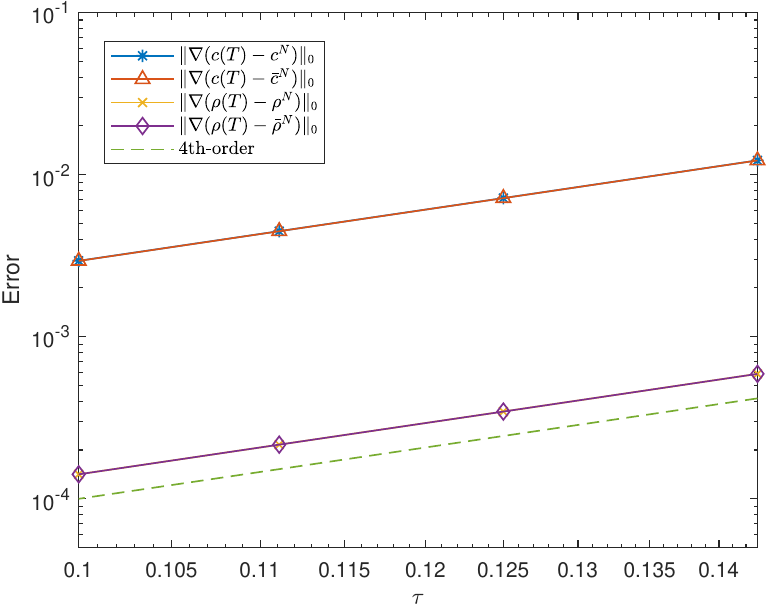}
			\caption{Fourth-order scheme}
		\end{subfigure}
		\begin{subfigure}{0.4\textwidth}
			\includegraphics[width=1\textwidth]{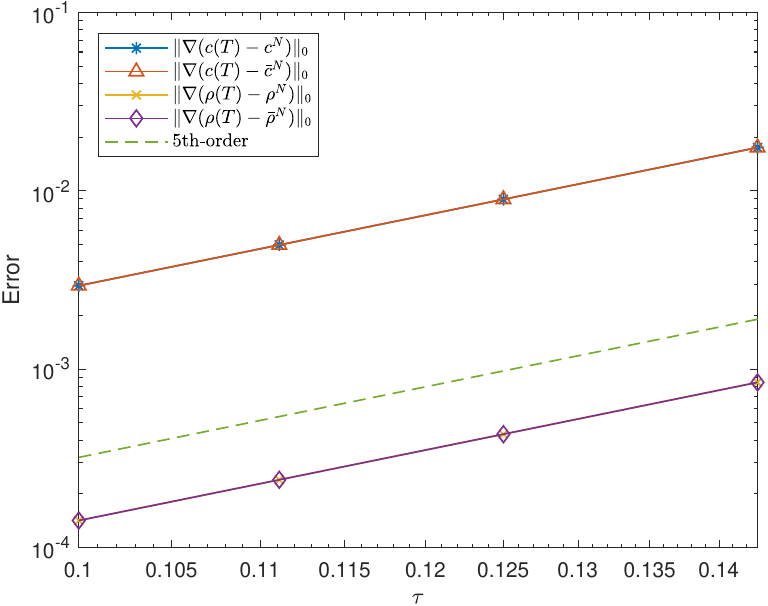}
			\caption{Fifth-order scheme}
		\end{subfigure}
		\caption{Convergence rates in  $H^1$-seminorm of the $k$th-order schemes.}\label{fig2}
	\end{figure}
	
	\begin{figure}
		\centering
		\begin{subfigure}{0.4\textwidth}
			\includegraphics[width=1\textwidth]{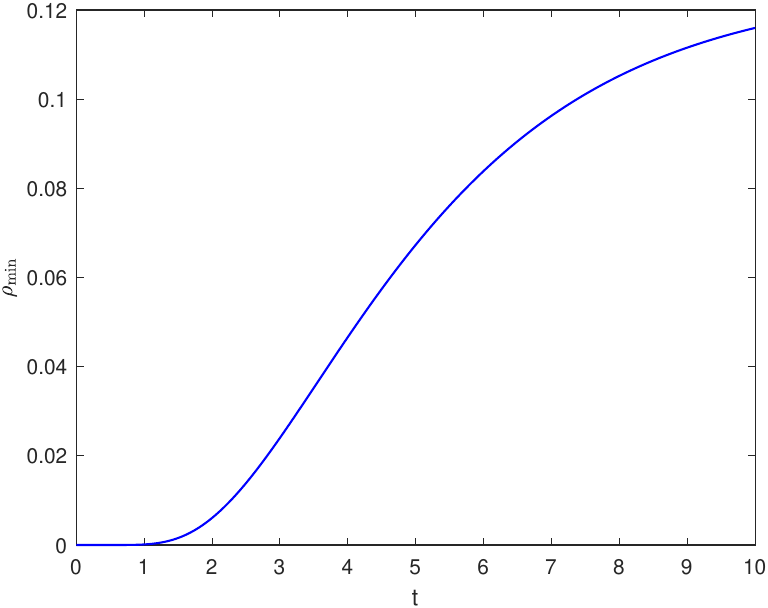}
			\caption{Minimum of $\rho^{n+1}$}
		\end{subfigure}
		\begin{subfigure}{0.4\textwidth}
			\includegraphics[width=1\textwidth]{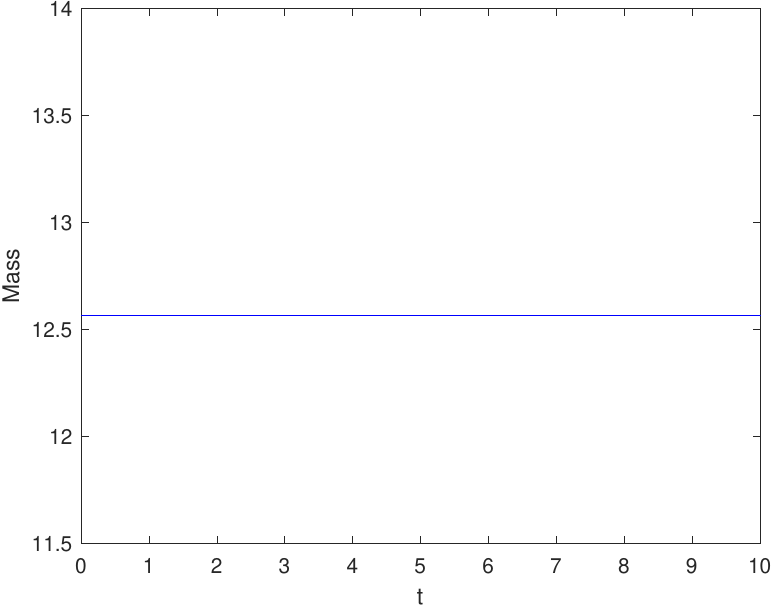}
			\caption{Evolution of cell mass}
		\end{subfigure}
		\begin{subfigure}{0.4\textwidth}
			\includegraphics[width=1\textwidth]{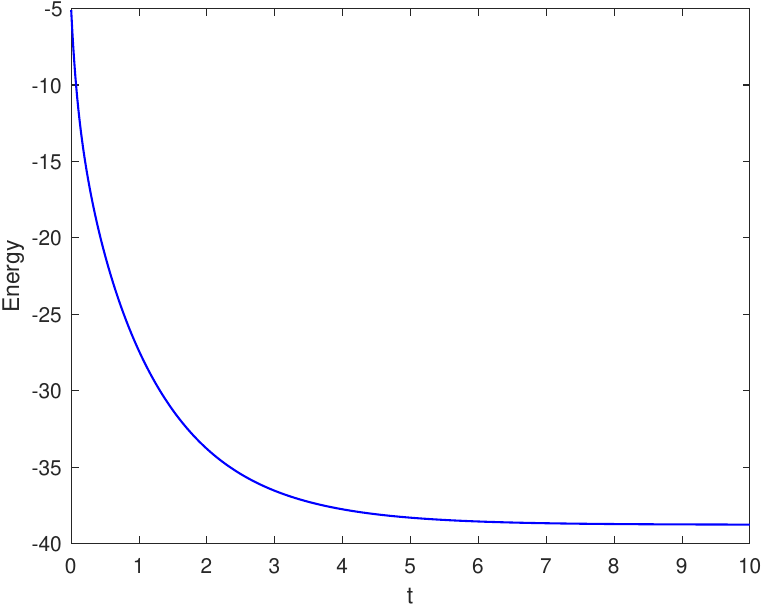}
			\caption{Evolution of numerical energy $E^{n+1}$}
		\end{subfigure}
		\caption{Simulations of the density, mass and energy (Example 2).}\label{fig3}
	\end{figure}
	
	\begin{figure}
		\centering
		\begin{subfigure}{0.4\textwidth}
			\includegraphics[width=1\textwidth]{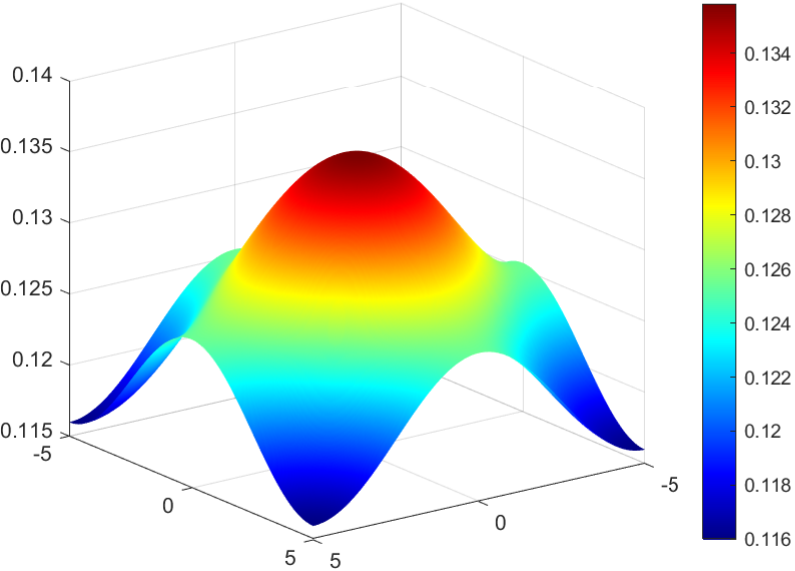}
			\caption{$\rho^{N}$}
		\end{subfigure}
		\begin{subfigure}{0.4\textwidth}
			\includegraphics[width=1\textwidth]{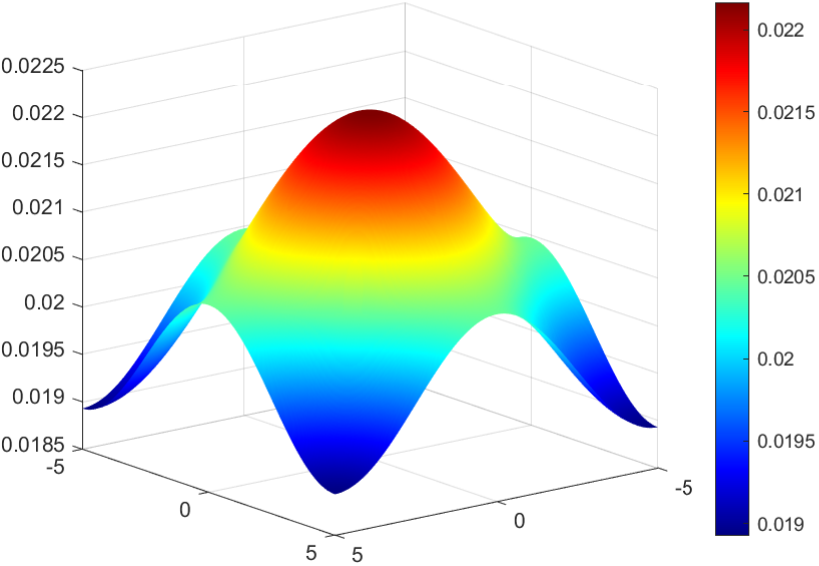}
			\caption{$c^{N}$}
		\end{subfigure}
		\caption{Numerical approximations of $\rho(t_{N})$ and $c(t_{N})$ (Example 2).}\label{fig4}
	\end{figure}
	
	\begin{figure}
		\centering
		\includegraphics[width=0.5\textwidth]{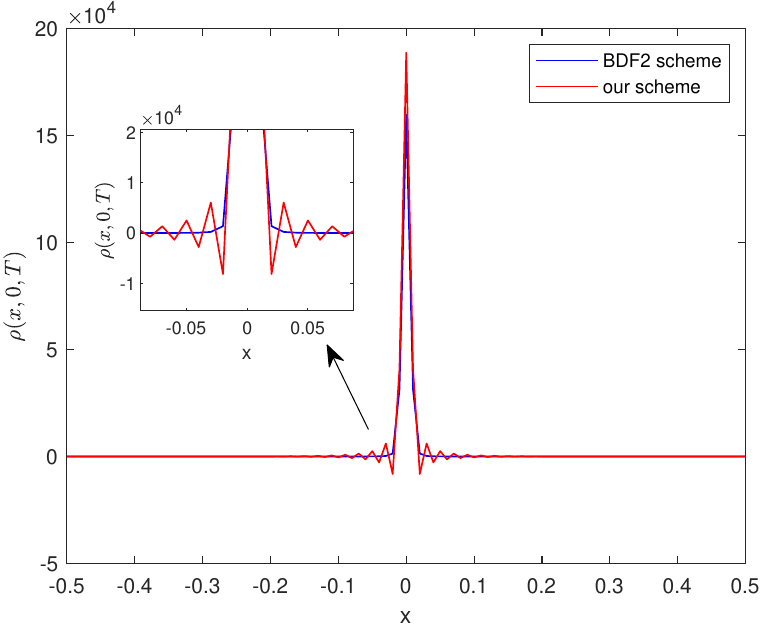}
		\caption{The cross section of $\rho(x, 0, T)$ with $T=1\times10^{-4}$ (Example 3).}\label{fig5}
	\end{figure}
	
	\begin{figure}
		\centering
		\begin{subfigure}{0.4\textwidth}
			\includegraphics[width=1\textwidth]{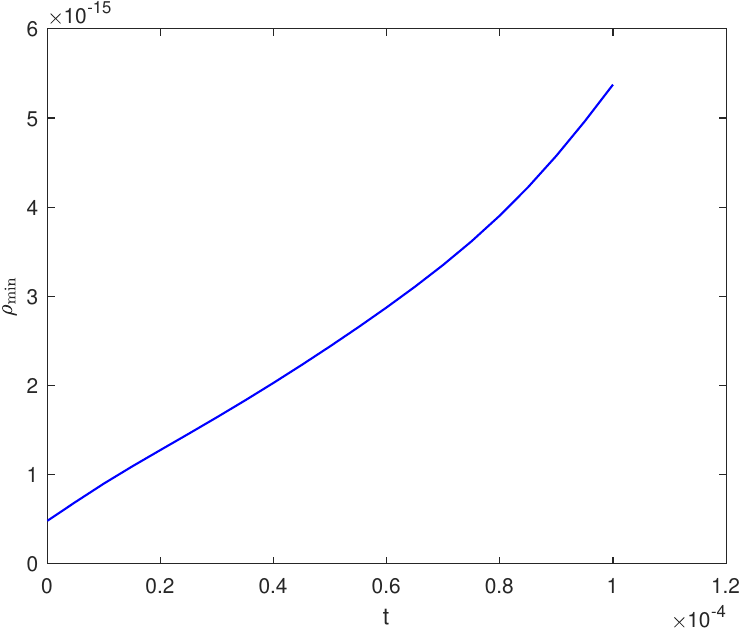}
			\caption{Minimum of $\rho^{n+1}$}
		\end{subfigure}
		\begin{subfigure}{0.4\textwidth}
			\includegraphics[height=0.83\textwidth]{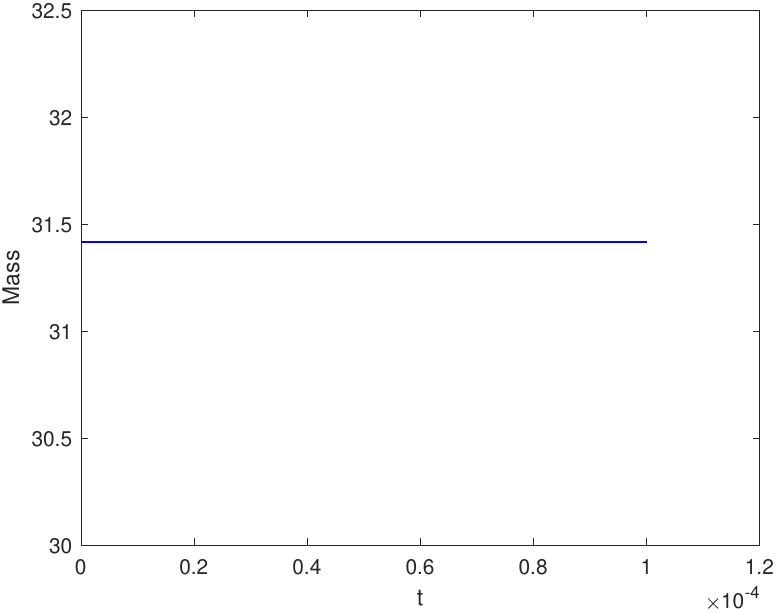}
			\caption{Evolution of cell mass}
		\end{subfigure}
		\begin{subfigure}{0.4\textwidth}
			\includegraphics[width=1\textwidth]{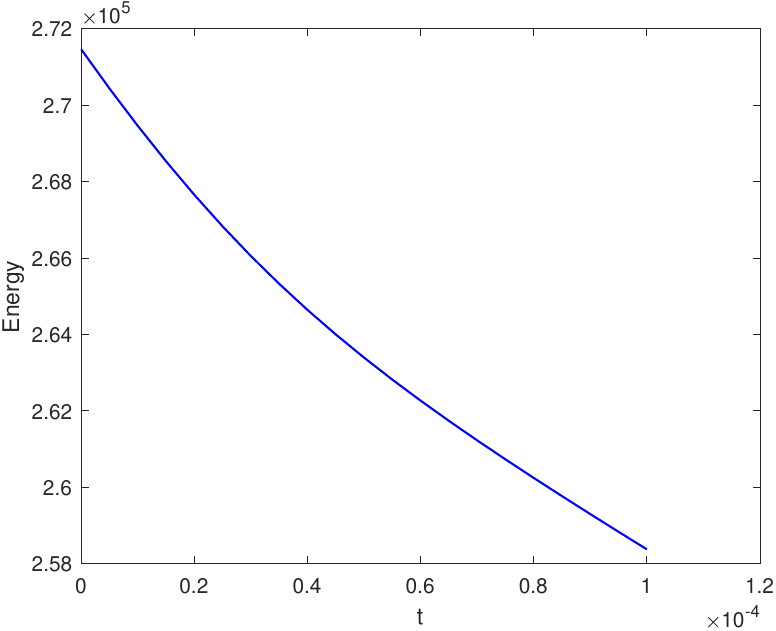}
			\caption{Evolution of numerical energy $E^{n+1}$}
		\end{subfigure}
		\caption{Simulations of the density, mass and energy (Example 3).}\label{fig6}
	\end{figure}
	
	\begin{figure}
		\centering
		\begin{subfigure}{0.4\textwidth}
			\includegraphics[width=1\textwidth]{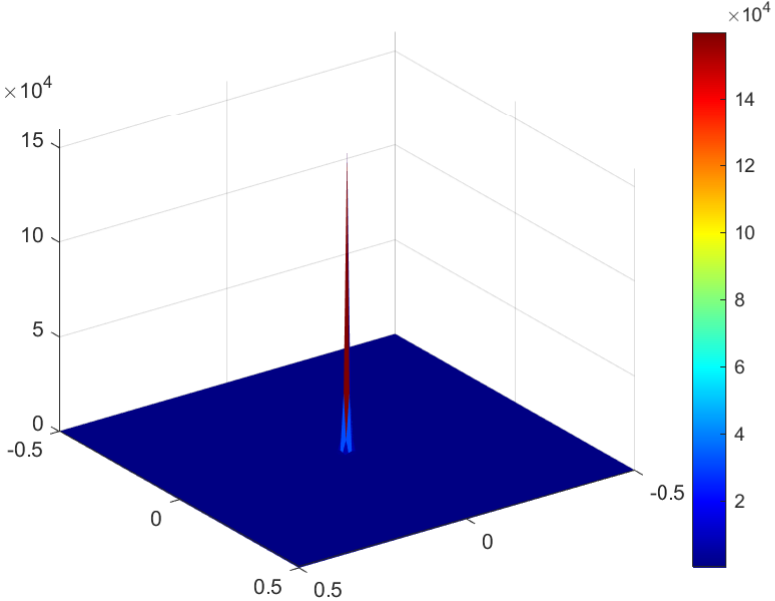}
			\caption{$\rho^{N}$}
		\end{subfigure}
		\begin{subfigure}{0.4\textwidth}
			\includegraphics[width=1\textwidth]{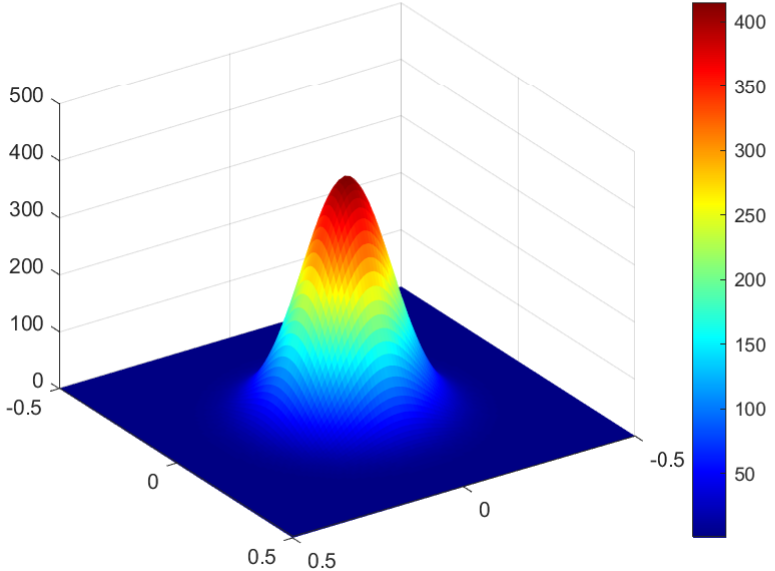}
			\caption{$c^{N}$}
		\end{subfigure}
		\caption{Numerical approximations of $\rho(t_N)$ and $c(t_N)$ (Example 3).}\label{fig7}
	\end{figure}
\end{document}